\documentclass[12pt]{amsart}
\pdfoutput=1
\usepackage[top=3cm, bottom=3cm, left=3cm, right=3cm]{geometry} 
\usepackage{graphicx}
\usepackage{color}
\usepackage{amssymb}

\newcommand{\A}{\alpha}

\newtheorem{theorem}{Theorem}[section]
\newtheorem{lemma}[theorem]{Lemma}

\theoremstyle{definition}
\newtheorem{definition}[theorem]{Definition}

\newtheorem*{restatemain}{Theorem \ref{theorem:MainTheorem}}
\newtheorem*{restatebounds}{Theorem \ref{theorem:ComplexityBounds}}

\theoremstyle{cor}

\theoremstyle{remark}
\newtheorem{remark}[theorem]{Remark}

\theoremstyle{remark}
\newtheorem{example}[theorem]{Example}

\theoremstyle{corollary}
\newtheorem{corollary}[theorem]{Corollary}

\numberwithin{equation}{section}

\begin{document}

\title{$\alpha$-sloped generalized Heegaard splitting of 3-manifolds}

\author{Marion Moore Campisi}
\address{Marion Moore Campisi, Department of Mathematics, University of Texas, Austin, 78712}
\email{campisi@math.utexas.edu}


\begin{abstract}
We generalize the definition of thin position of Scharlemann and Thompson for compact orientable 3-manifolds with torus boundary components and introduce $\alpha$-sloped generalized Heegaard splittings.   We examine its relationship to generalized Heegaard splittings of manifolds resulting from Dehn filling. We compare $\alpha$-sloped thin position of 3-manifolds to other types of thin position for knots and 3-manifolds and discuss how this kind of decomposition gives an organic picture of M and allows the structure of the manifold to dictate the most natural slope(s) on the boundary.   Additionally, we provide illustrative examples and questions motivating the study of alpha-sloped thin position. 
\end{abstract}

 \maketitle
 
\begin{center}
\end{center}

        \section{Introduction}
        \label{sec:Introduction}

The notion of thin position for a closed orientable 3-manifold was developed by Scharlemann and Thompson in \cite{SchaThoTPf3M}.  Their construction can also be applied to compact orientable 3-manifolds by regarding a compact 3-manifold as a cobordism from one (possibly empty) collection of boundary components to the collection of remaining components.  The construction is to build a manifold by starting with a set of 0-handles and (possibly) a closed surface cross an interval and then attaching alternating collections of 1- and 2-handles and finally adding 3-handles.  The 1- and 2-handles are added so as to minimize the complexity of the intermediate surfaces.  This decomposition along closed surfaces is called a generalized Heegaard splitting and a generalized Heegaard splitting of lowest complexity is called thin position.  We generalize this construction to decompositions of compact 3-manifolds with a collection of torus boundary components called $\alpha$-sloped generalized Heegaard splittings.

In an $\alpha$-sloped generalized Heegaard splitting a manifold $M$ is decomposed into simple pieces along surfaces which have boundary with slope $\alpha \epsilon \{\mathbb{Q} \cup \infty \cup \emptyset\}$ on a distinguished boundary component of $M$.  Assigning a complexity to these decompositions, \emph{$\alpha$-sloped thin position} is a decomposition which has lowest complexity over all $\alpha$-sloped decompositions, and  \emph{thin position} of $M$ is a decomposition of $M$ with the lowest complexity over all slopes.  This kind of decomposition gives an organic picture of $M$ and allows the structure of the manifold to dictate the most natural slope(s) on the boundary.   In section \ref{section:PropertiesofThinDecompositions} we show that $\alpha$-sloped generalized Heegaard splittings have many of the same desirable properties as classical generalized Heegaard splittings.

 
 It is always possible to decompose a manifold with torus boundary using $\alpha$-sloped surfaces.  Sometimes the only possible way to do this efficiently is to modify a closed decomposition via a process called $\alpha$-stabilization.  Slopes for which this is the only possibility are called \emph{dishonest} and all others are called \emph{honest}.  Boundary slopes for separating essential surfaces and slopes that lead to a drop in Heegaard genus under Dehn surgery are honest slopes, and this is the class of slopes for which this type of analysis is most interesting.  


Using $\alpha$-sloped generalized Heegaard splittings we are able to  compare decompositions of knot exteriors to decompositions of the manifolds which result from Dehn filling.  In section \ref{sec:InducedHSplittings} we establish that for knots in $S^3$ unless a thin $\alpha$-sloped decomposition is comprised of a single surface (i.e. is a $\alpha$-sloped Heegaard splitting) then either $S^3-K$ or the manifold resulting from $\alpha$-sloped Dehn surgery on $K$ contains an closed essential surface.  Specifically we prove:

\begin{theorem}
\label{theorem:MainTheorem}

Let $K \subset S^3$ be a non-trivial knot and let $\alpha \epsilon \{\mathbb{Q} \}$.  Let $M=\overline{S^3-K}$ be in $\alpha$-sloped thin position, and let $K(\alpha)$ be the result of $\alpha$-sloped Dehn surgery on $K$.  Then either 
\begin{enumerate}

\item{$\A$-sloped thin position for $M$ is an $\A$-sloped Heegaard surface }

\item{There exists a closed essential surface in the exterior of $K$}
\label{item:closed}

\item{$K(\alpha)$ is Haken}
\label{item:Haken}

\item{$K(\alpha)$ is a connected sum of two lens spaces}.
\label{item:lensspace}

\end{enumerate}
\end{theorem}

In section \ref{sec:DefinitionsAndConstruction} we describe a method for modifying a closed surface to create a surface with slope $\A$, and for modifying an $\alpha$-sloped surface to create a closed surface.  Extending this to a whole decomposition, we see that for a given manifold $M$ the width of any slope, $w(M, \alpha)$, is bounded by a function of the closed width, $w(M, \emptyset)$, which is similar to the width of the thin decomposition of Scharlemann and Thompson.  In fact, the closed width is close to an upper bound for any $\alpha$-width.  Moreover, the higher the closed width, the wider the range of possible $\alpha$-widths over all $\alpha$.    In section \ref{sec:ComplexityBounds} we show:

\begin{theorem} 
\label{theorem:ComplexityBounds}

 For any $\alpha \epsilon \{\mathbb{Q} \cup \infty\}$,   $\lceil \frac{2}{3} \times w(M, \emptyset)\rceil \leq  w(M, \alpha) \leq w(M, \emptyset) +_1 2 $.
\end{theorem}

$\lceil \frac{2}{3} \times w(M, \emptyset)\rceil$ and $w(M, \emptyset) +_1 2 $ are functions of $w(M, \emptyset)$ described in section \ref{sec:ComplexityBounds}.

In section \ref{sec:TorusKnots} we classify torus knot decompositions and show that this bound on the complexity range is sharp and that the width of Torus knots is realized by both the lens space surgery slopes and the meridional slope.

    \section{Definitions and Construction}{}{}
    \label{sec:DefinitionsAndConstruction}
    
    Throughout let $M$ be a compact orientable manifold whose boundary is a non-empty collection of torus boundary components, $\{T_i\}$. 
    
    Surfaces with boundary in such manifolds can be modified using the boundary of $M$ in a natural way to lower the number of boundary components of the surface.

Choose a boundary torus $T$ and consider the complement of $\Sigma$ in $T$,  $T \setminus \Sigma$ for some properly embedded surface with boundary $\Sigma \subset M$ with $\partial \Sigma \cap T \neq \emptyset$.  $T \setminus  \Sigma$ is a collection of annuli $\{A_i\}$.  

We define \emph{a tube of $\Sigma$ along $T$} to be a surface $\Sigma '= \Sigma - [A_j\times I] \cup[ A_j \times \{1\}] $.  



    \subsection{Boundary Compression Bodies}
   \label{subsec: BoundaryCompressionBodies}

A \emph{compression body} is a connected 3-manifold obtained from a closed (possibly disconnected) surface, $\partial_- C$, by adding a collection of 1-handles, $\mathcal{H}^1$, to $\partial_- C \times \{1\}$ in $\partial_- C \times I$. Let  $\partial_+ C$ denote $ \partial C - \partial_- C$.  Dually, a compression body is obtained from a connected surface $\partial_+ C$ by attaching 2-handles to $\partial_+C \times \{0\}$ in $\partial_+ C \times I$ and adding 3-balls onto newly-created 2-spheres.



A natural generalization of this construction results in a boundary compression body.  Define a \emph{bead} to be a solid torus $b=A \times I$, for an annulus $A$.  Let $\partial_0b$ denote $A\times\{0\}$ and $\partial_+b$ denote $\partial b - \partial_0 b$.  A \emph{spanning arc} for $b$ is an arc in $\partial_0b$ connecting its two boundary components.  The \emph{core} $c$ of $b$ is an essential curve in $\partial_0b$ and the \emph{co-core} is $[s\times I] \subset [A \times I]$ for a spanning arc $s$.  A \emph{0-bead} is a bead with empty attaching region, and a \emph{2-bead} is a bead with attaching region $\partial_+ b$.

 A \emph{boundary compression body} $C$ is a connected 3-manifold obtained from the disjoint union of a (possibly empty) collection of 0-beads and 0-handles, and $\partial_- C \times I$ for a (possibly disconnected) surface (possibly with boundary), $\partial_- C$, by attaching a collection of 1-handles to $[\partial_- C \times \{1\}] \cup \partial_+ b^0$.    Denote the collection of all 0-handles $\mathcal{H}^0$, the collection of 1-handles $\mathcal{H}^1$ and the collection of 0-beads $b^0$.
 
The boundary of the boundary compression body is divided into three section; $\partial_-C$, $\partial_0 C =[\partial (\partial_-C) \times I] \cup \partial_0b^0$ and $\partial_+ C= \partial C - (\partial_- C \cup \partial_0 C)$, see Figure  \ref{figure:boundarycompbody}.  Note that $\partial_0 C$ is a collection of annuli.  A compression body is a boundary compression body with $\partial(\partial_-C)= \emptyset$ and $b^0=\emptyset$. 


Dually a boundary compression body $C$ is obtained from a connected surface (possibly with boundary), $\partial_+C$, by adding a collection of 2-handles, $\mathcal{H}^2$, to $\partial_+ C \times \{0\}$ in $\partial_+ C \times I$ and attaching a collection of 2-beads, $b^2$, to annular components of the surface $\partial [(\partial_+ C \times I \cup \mathcal{H}^2)-(\partial_+C \cup \partial (\partial_+C\times I))]$ and capping off spherical boundary components with 3-handles.  As above the boundary is divided into $\partial_+C$, $\partial_0 C =[\partial (\partial_+C) \times I] \cup \partial_0b^2$ and $\partial_- C= \partial C - (\partial_+ C \cup \partial_0 C)$.  

\begin{figure}[htbp]
\begin{center}
\includegraphics[height=2in]{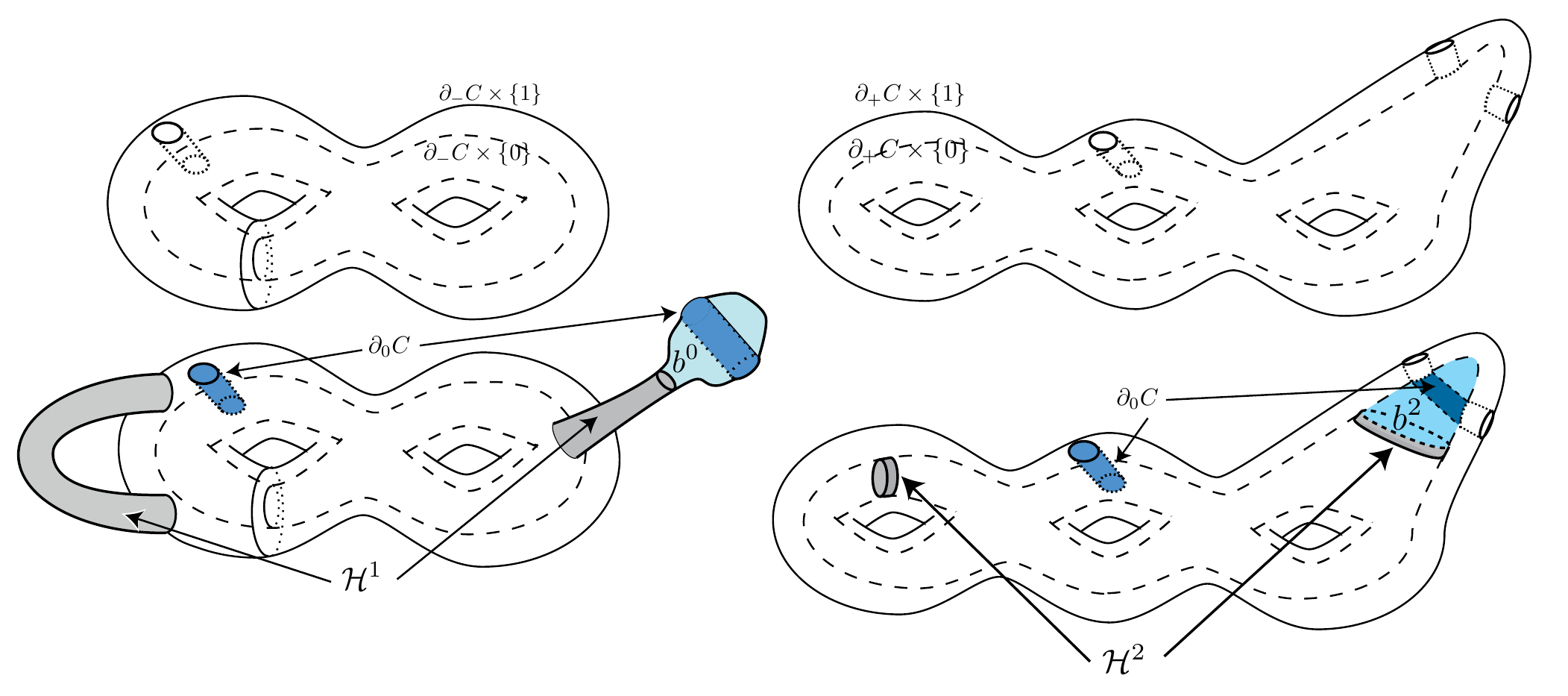}
\caption{The construction of a boundary compression body.}
\label{figure:boundarycompbody}
\end{center}
\end{figure}


A \emph{cut system} for a boundary compression body $C$  is a collection of disjoint, non-parallel compressing disks $\mathcal{D}$ for $C$, such that $C | \mathcal{D}$ is $( \partial_- C \times I)\cup b^0$.  If $\partial_- C=\emptyset$ then $C | \mathcal{D}$ is a collection of beads $b^0$, unless $b^0=\emptyset$, in which case it is a ball.  

Let $C$ be a boundary compression body.  A surface $F\subset C$ is \emph{$\partial_0$-compressible} in $C$ if there exists a disk $D\subset C$ with $\partial D = \gamma \cup \beta$ where $\gamma$ is an essential arc in $F$ and $\beta$ is an essential arc in $\partial_0 C$.


A 3-manifold $M$ is \emph{irreducible} if every embedded $S^2$ bounds a 3-ball.  Topologically, boundary compression bodies are compression bodies, so they are irreducible.  

If $M=C_1 \cup_S C_2$ for a pair of boundary compression bodies $C_1$ and $C_2$ we call $S$ a \emph{boundary Heegaard surface} for $M$ and the decomposition $M=C_1 \cup_S C_2$ a \emph{boundary Heegaard splitting of $M$}.

\begin{lemma}In a boundary compression body $C$, $\partial_-C$ is incompressible and $\partial_0$-incompressible. 
\label{lemma:negboundaryincomp} 
\begin{proof} 
Let $\mathcal{H}^1$ and  $b^0$ be as above.  
Let $C=[\partial_-C \times I] \cup \mathcal{H}^1 \cup b^0$,  and let $\mathcal{D}$ be a  cut system for $C$. Assume there is a compressing disk or $\partial_0$-compressing disk  $D\subset C$  for $\partial_-C$.  If $D$ is a $\partial_0$-compressing disk then $\partial D = \gamma \cup \beta$ where $\gamma$ is an essential arc in $\partial_- C$ and $\beta$ is an essential arc in $\partial_0 C$.    Since $\partial_0 C$ is the collection of annuli $[\partial (\partial_-C) \times I] \cup \partial_0b^0$, $\beta$ must be a spanning arc for some component $A$ of $\partial_0 C$.  Since $\partial_0b^0 \cap \partial_- C = \emptyset$ and $\partial \beta = x \cup y = \partial \gamma \subset \partial_-C$, $A$ cannot be an element of $\partial_0b^0$.  Thus $A=c \times I\subset \partial(\partial_-C) \times I$ for some boundary component $c$ of $\partial(\partial_-C)$.  It follows that either $x$ or $y\subset c \times \{1\}$, a contradiction.  Thus $D$ is not a $\partial_0$-compressing disk, but a compressing disk.   

Consider $\mathcal{D} \cap D$. Using an innermost disk argument, since $C$ is irreducible it is possible to remove circles of intersection. Since $\partial D \subset \partial_-C$ and $\partial \mathcal{D} \subset \partial_+ C$, there are no arcs of intersection and thus $D$ can be isotoped to be disjoint from $\mathcal{D}$.  Thus either $D \subset \partial_- C \times I$ or $D\subset b^0$.  Since $\partial D \subset \partial_- C$ the latter is impossible so $D$ is a compressing disk for $\partial_- C$ in $\partial_- C \times I$, a contradiction.  
\end{proof}

\end{lemma}

\begin{lemma}
\label{lemma:boundaryplusboundarycomp}

In a boundary compression body $C$, any $\partial_0$-compressing disk $D$ for $\partial_+ C$ is isotopic to the co-core of some bead $b$.


\begin{proof}
Let $D$ be a $\partial_0$-compressing disk for $\partial_+C$.  The boundary of $D$ is made up of two arcs, $\gamma \subset \partial_+ C$ and $\beta \subset \partial_0 C$. Since $\partial \gamma = \beta \subset \partial_+C$,  $\beta$ must be contained in $\partial_0 b$ and is the spanning arc for some bead $b$.  Consider $D \times  [-1, 1]$. Let $D_1$ be $D \times \{-1\}$ and  let $D_2$ be $D \times \{1\}$ with $D\times \{0\}=D$.  The pair of disks $D_1 \cup D_2$ cuts $\partial_0 b$ into two rectangles, one of which contains $\beta$.  Call the other rectangle $R$. Choose $\epsilon$ small enough so that $R \times [0, \epsilon] \cap D_i$ is a single rectangle for $i=1,2$.  Let $D_i'$ be $D_i \setminus (R \times [0, \epsilon])$ for $i=1,2$. The disk $D'=D_1' \cup D_2' \cup (R \times \{ \epsilon \})$ is a compressing disk for $C$.  $C|D'$ has two components, one of which contains $\partial_0b$.   Call this component $C'$.  $C'=D\times  [-1, 1] \bigcup R \times [0, \epsilon]$ and is a bead with co-core $D$.  
\end{proof}

\end{lemma}

\subsection{Decomposing link exteriors}

In order to decompose manifolds with torus boundary components into boundary compression bodies, it is helpful to reframe the definition of a bead in terms of torus boundary.  For a link $L\subset S^3$ we will abuse notation and denote decomposition of $\overline{S^3 - n(L)}$, the exterior of $L$ in $S^3$, as decompositions of $L$.

In what follows let $M$ be a manifold with toral boundary components $\{T_i\}$.  Select a boundary component $T \epsilon \{T_i\}$.  We  parametrize slopes on $T$ by elements of  $\{\mathbb{Q} \cup \infty\}$, as in \cite{BoyDSoK}.  

\begin{definition}  Given a 2-dimensional handle decomposition of $T$, a $\emph{0-bead}$,  $b^0 \subset M$ is a regular neighborhood in $M$ of the union of a 2-dimensional 0-handle and a 2-dimensional 1-handle with slope $\alpha \epsilon \{\mathbb{Q} \cup \infty\}$ in $T$, and a $ \emph{2-bead}$,  $b^2 \subset M$ is a regular neighborhood in $M$ of the union of a 2-dimensional 1-handle with slope $\alpha \epsilon \{\mathbb{Q} \cup \infty\}$ and a 2-dimensional 2-handle in $T$, see Figure \ref{figure:bead}.  Let $\alpha \epsilon\{\mathbb{Q} \cup \infty\}$.  We call a bead $b^i$, $i=0,2$, an $\alpha$-sloped bead, if the associated 2-dimensional 1-handle has slope $\alpha$.  In this context $\partial_0 b^i$ denotes $b^i \cap T$ and $\partial_+ b^i$ denotes $\partial b^i - \partial_0 b^i$, see Figure \ref{figure:bead}.  
\end{definition}

\begin{figure}[htbp]
\begin{center}
\includegraphics[height=1.75in]{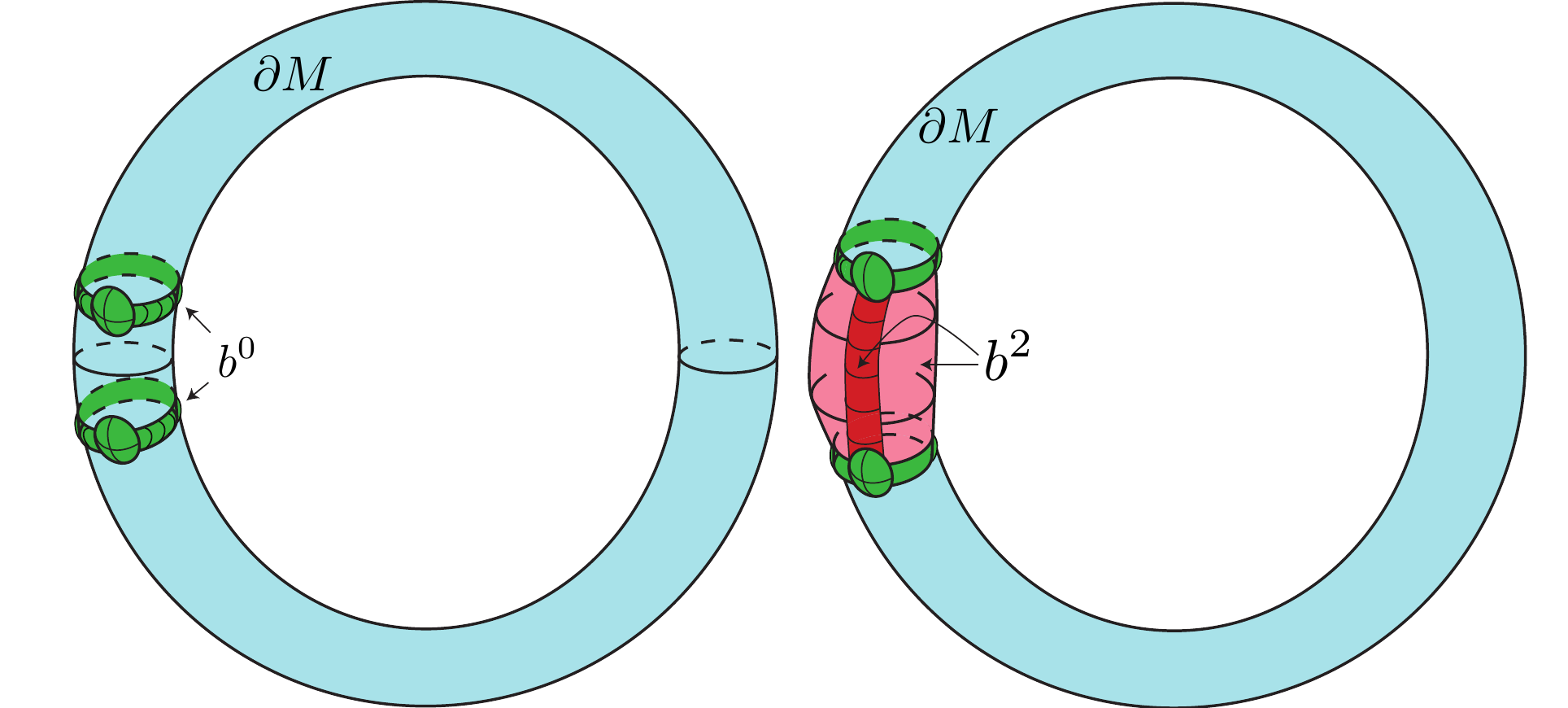}
\caption{A bead, and a 2-bead}
\label{figure:bead}
\end{center}
\end{figure}

Let $H$ be a compression body properly embedded in a 3-manifold with boundary $M$. If $\partial M\cap H \neq \emptyset$ then $\partial M \cap H \subset \partial_-H$.  In contrast, if $C$ is a boundary compression body with boundary slope $\alpha$ on a distinguished boundary component $T$ of $M$ then $T \cap C \neq \emptyset$ and $T \cap C \subset \partial_0 C$.  Note that if $C$ contains an $\alpha$-sloped bead on $T$ then $C$ has boundary slope $\alpha$.  

For any $\alpha \epsilon \{\mathbb{Q} \cup \infty\}$ $M$ can always be decomposed along a surface $F$ with boundary slope $\alpha$ on a distinguished boundary component $T$ into two boundary compression bodies.  We call such a decomposition an \emph{$\alpha$-sloped Heegaard splitting}.  To see this we modify a (closed surface) Heegaard splitting for $M$ to create an $\alpha$-sloped Heegaard splitting via a process called \emph{$\alpha$-stabilization}.

First we establish some notation.  Fix $\alpha\epsilon \{\mathbb{Q} \cup \infty\}$.

In any compression body $C$ with a torus component $T$ of $\partial_-C$ there exists a properly embedded annulus  $A$ with one boundary component on $\partial_+C$ and one on $T$ with slope $\alpha$.  Call this annulus \emph{$A_{\alpha}$, the $\alpha$-sloped spanning annulus for $T$}.

\begin{definition}
\label{definition:alphastab}

Let $C_1 \cup_S C_2$ be a Heegaard splitting of a compact orientable manifold M with distinguished torus boundary component $T\subset C_2$.  An \emph{$\alpha$-stabilization} of $C_1 \cup_S C_2$ is an $\alpha$-sloped Heegaard splitting for $M$ which results from the following:

 Let $c$ be a curve with slope $\alpha$ in $T$, let $*$ be a point in $c$, and let $\gamma$ be a properly embedded arc in $C_2$ connecting $S$ to $c$ which is unknotted in the sense that $C_2 - \gamma$ is a compression body.   Add a neighborhood of $c$ along with a neighborhood of $\gamma$ to $C_1$ and delete them from $C_2$.   The result is to transform $C_1$ into a boundary compression body $C_1'$ by adding a 0-bead $b$ with core $c$  and a one handle with core $\gamma$ to $C_1$.

Pick a spanning annulus $A_{\alpha}$  with the property that $A_{\alpha}\cup n(\gamma)$ is a single disk and  $A_{\alpha} - n(\gamma)$ is a single disk.    The compression body $C_2$ is also transformed into a boundary compression body $C_2'$ by adding a two handle with attaching curve $\partial (A_{\alpha} - n(\gamma))$ and a 2-bead $b'$ whose core is an essential curve in the annulus $T-c$.  See Figure \ref{figure:TXI}.
\end{definition}

\begin{figure}[htbp]
\begin{center}
\includegraphics[height=2in]{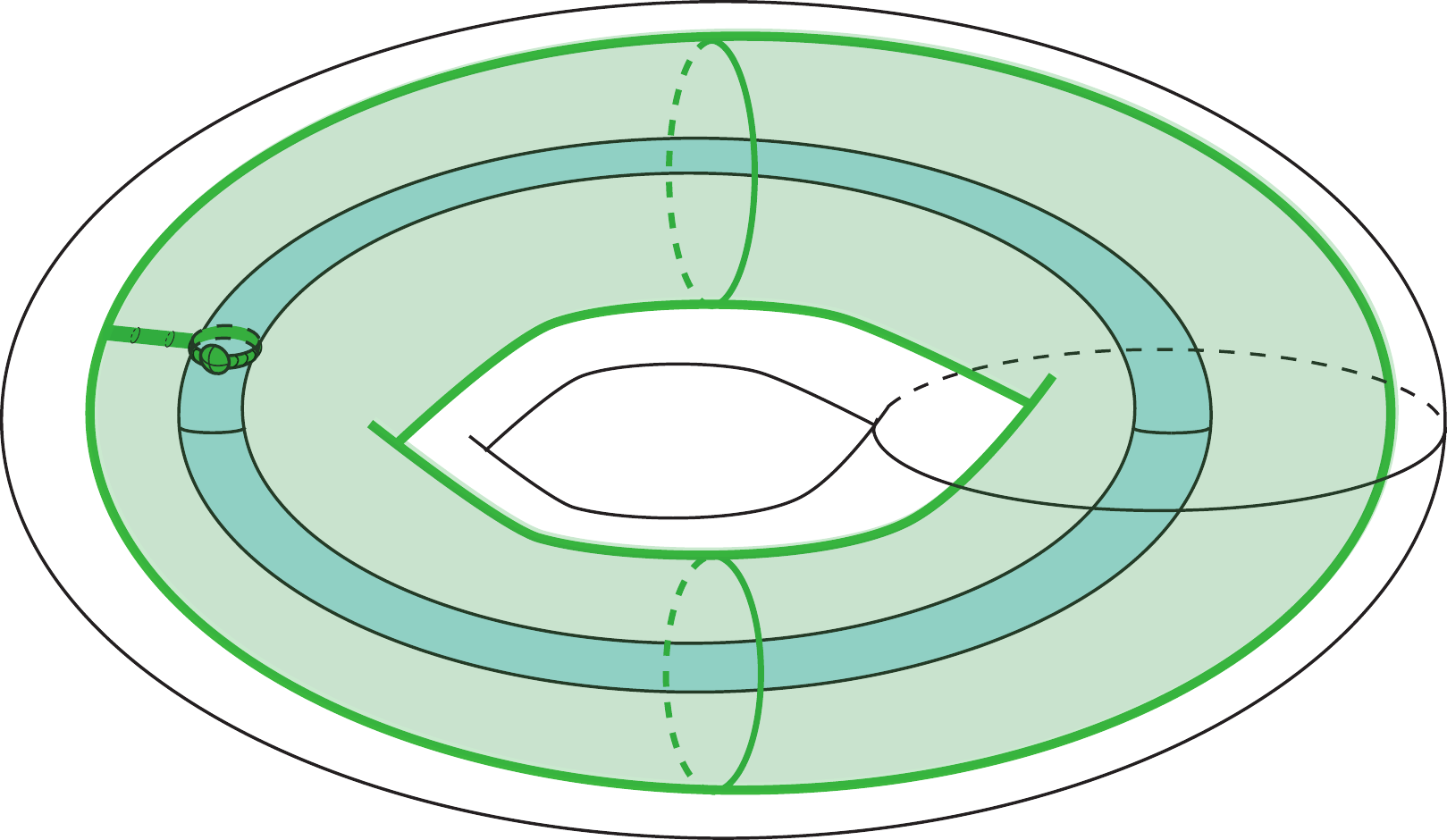}
\caption{An $\alpha$-stabilization of a Heegaard splitting of $T^2 \times I$}
\label{figure:TXI}
\end{center}
\end{figure}

%
%
%

%

The $\alpha$-stabilization of a closed Heegaard surface results in an $\alpha$-sloped Heegaard surface of the same genus, with two boundary components.   Note that any  boundary Heegaard surface can be $\alpha$-stabilized in the same way, resulting in a boundary Heegaard surface of the same genus with two more boundary components.

Observe that given an $\alpha$-stabilized Heegaard surface $S$, there exists a separating compressing disk $D^* \subset C_2'$ for $S$ such that one component of $S | D$ is a closed Heegaard surface, and the other component is a boundary parallel annulus with slope $\A$.

\begin{lemma}
\label{lemma:alphastabcomptoclosed} 

Let $M$ be a compact orientable manifold with torus boundary.  An $\alpha$-sloped Heegaard splitting $S$ of $M$ along $T$ is $\alpha$-stabilized if and only if there is a compressing disk $D$ for $S$ such that $S | D$ is the disjoint union of a Heegaard surface for $M$ which is either closed or $\alpha$-sloped and an $\A$-sloped boundary parallel annulus.  

\begin{proof}
Let $\gamma$ be as in the definition of $\alpha$-stabilization.  If a splitting is $\alpha$-stabilized then a meridian for $n(\gamma)$ is such a compressing disk.    

Suppose that for some disk $D$,  the result of compressing $S$ along $D$ is the disjoint union of a closed splitting $S'$ for $M$, and an $\A$-sloped boundary parallel annulus $A$.  Let $C_1$, $C_2$ be the boundary compression bodies bounded by $S$, and $C_1'$, $C_2'$ the compression bodies bounded by $S'$.  For the sake of notation, assume $T\subset  C_1'$.  Thus $\partial A \subset C_1'$ and $A$ is boundary parallel into $T$.  The solid torus realizing the boundary parallelism is a bead $b$ contained in $C_1'$.  It is possible to isotope $A$ and $S'$ so that $A \cap S'$ is the surgery disk $D$.  

Now $C_2 \cap (A \times I)=D$ and the pre-surgery surface $S$ is a twice punctured torus bounding $C_2'$.  But $C_2$ can also be realized as a boundary compression body obtained from attaching $C_2$ to the $\A$-sloped bead $b$ by a 1-handle which is dual to $D$ and $S$ is an $\alpha$-stabilized Heegaard splitting of $M$.

\end{proof}

\end{lemma}

An $\alpha$-sloped Heegaard splitting gives rise to a handle and bead decomposition of $M$: $M=\mathcal{H}^0 \cup b^0 \cup \mathcal{H}^1 \cup \mathcal{H}^2 \cup b^2 \cup \mathcal{H}^3$.   Given such a decomposition,  in some cases the additions of handles and beads can be reordered, leading to different $\alpha$-sloped decomposition of $M$.  $M=\mathcal{H}^0 \cup {b^0}_1 \cup {\mathcal{H}^1}_1 \cup {\mathcal{H}^2}_1 \cup {b^2}_1 \cup {b^0}_2 \cup {\mathcal{H}^1}_2 \cup {\mathcal{H}^2}_2 \cup {b^2}_2 \cup ... \cup {b^0}_k \cup {\mathcal{H}^1}_k \cup {\mathcal{H}^2}_k \cup {b^2}_k \cup \mathcal{H}^3$.  

We consider the boundary of the series of submanifolds resulting from each addition of a handle or bead. A natural collection of interesting surfaces in $M$ arises.  We define two classes of surfaces given by a specific decomposition of $M$, the \emph{thick} and \emph{thin} surfaces.  Roughly  the \emph{thick surfaces}, $S_i$, are the boundaries of submanifolds which are the result of adding an entire collection of 0-beads ${b^0}_i$ and 1-handles  ${\mathcal{H}^1}_i$ and the \emph{thin surfaces}, $F_i$, are the boundaries of submanifolds which result from adding an entire collection of 2-handles  ${\mathcal{H}^2}_i$ and 2-beads ${b^2}_i$.  More precisely:

\begin{definition}
Let the $i$th \emph{thick surface} $S_i$, $1 \leq i \leq k$, be the surface obtained from $\partial [{H}^0 \cup {b^0}_1 \cup {\mathcal{H}^1}_1 \cup {\mathcal{H}^2}_1 \cup {b^2}_1 \cup ... \cup {b^0}_i \cup {\mathcal{H}^1}_i]$ by deleting all spheres which bound 0- or 3-handles and all boundary parallel annuli which bound beads or 2-beads in the decomposition.
\end{definition}

\begin{definition}
  Let the $i$th \emph{thin surface} $F_i$, $1 \leq i \leq k-1$ be the surface obtained from $\partial [{H}^0 \cup {b^0}_1 \cup {\mathcal{H}^1}_1 \cup {\mathcal{H}^2}_1 \cup {b^2}_1  \cup ... \cup {b^0}_i \cup {\mathcal{H}^2}_i]$ by deleting such spheres and annuli.  
  \end{definition}

  \begin{figure}[htbp]
\begin{center}
\includegraphics[height=2in]{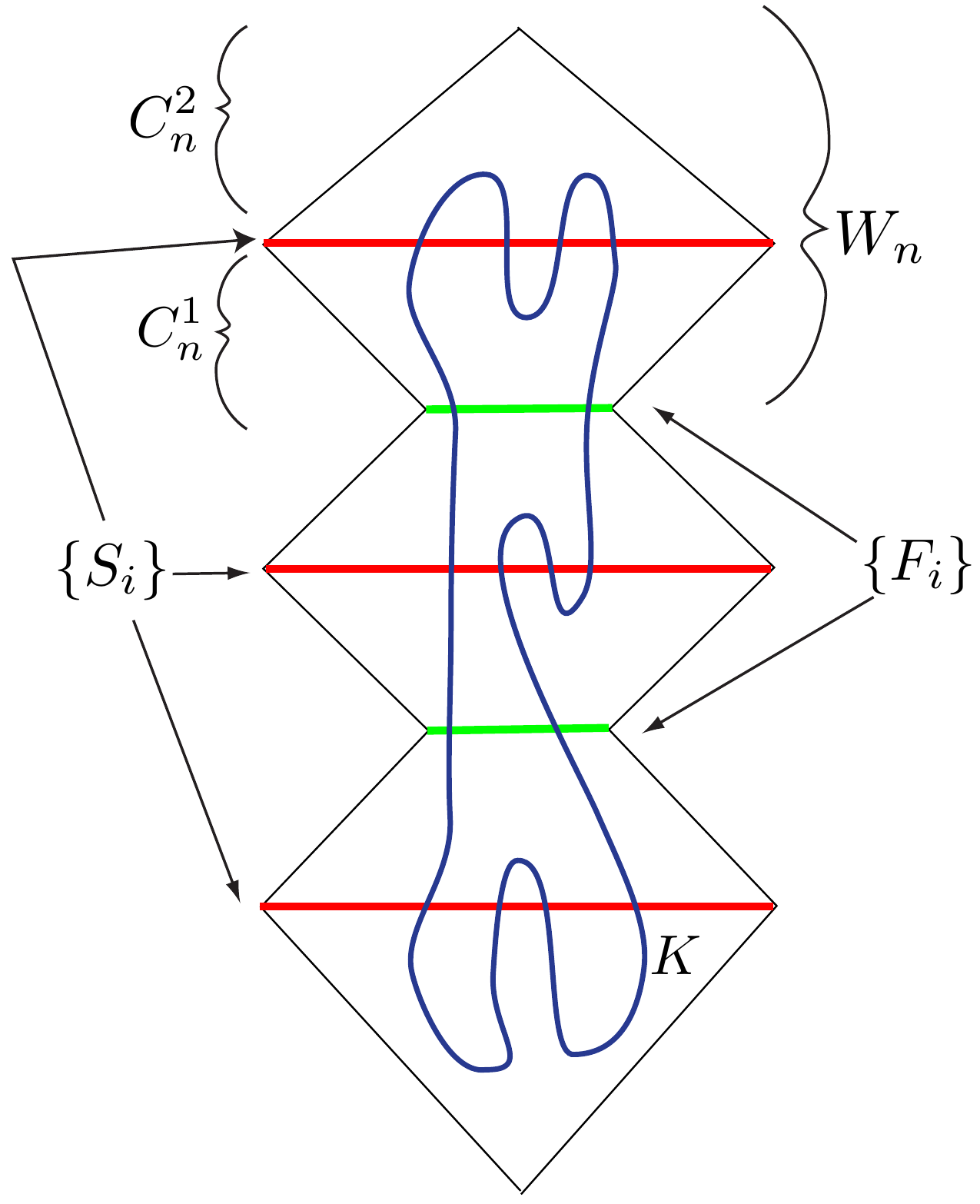}
\caption{An $\alpha$-sloped decomposition of $M$ cuts $M$ into simple pieces.}
\label{figure:surfacelabels}
\end{center}
\end{figure}

We call the collection of surfaces $\{F_i,S_i\} \subset M$ an \emph{$\alpha$-sloped decomposition of $M$ along $T$} or  an \emph{$\alpha$-sloped generalized Heegaard splitting for $M$ along $T$}.  If $M$ has a single boundary component, or the relevant boundary component is clear, we simply call it an \emph{$\alpha$-sloped generalized Heegaard splitting of $M$}.   
  
Let $W_i=(F_{i-1} \times I)\cup {b^0}_i \cup  {\mathcal{H}^1}_i \cup  {\mathcal{H}^2}_i \cup  {b^2}_i$ together with any 0- or 3-handles and any beads or 2-beads incident to  ${\mathcal{H}^1}_i$, ${b^1}_i$,   ${\mathcal{H}^2}_i$, and ${b^2}_i$. Then $M | \{F_i\}$ is the collection of 3-manifolds $W_i$ with Heegaard splittings $W_i =C_i^1\cup_{S_i} C_i^2$, see Figure \ref{figure:surfacelabels}.  Two $\A$-sloped decompositions $\{F_i,S_i\} , \{F_i',S_i'\}$ are isotopic if $F_i$ is isotopic to $F_i'$  and $S_i$ is isotopic to $S_i'$ in $M$ for all $i$.   


\begin{definition}Given a connected surface with  boundary  $S$ of genus $g>0$, properly embedded in $M$, the complexity of $S$, $c(S)$, is given by $c(S)=1-\chi(S)+g(S)$. Define $c(S^2)=0$. If $S$ is disconnected,  $c(S)=\sum \{c(S^*)| S^*$ is a connected component of $S\}$.

\end{definition}

Let $\{F_i,S_i\}$ be an $\alpha$-sloped decomposition of $M$.  We define the \emph{width of $\{F_i,S_i\}$}, $w(M, \{F_i,S_i\})$, to be the set of integers $\{c(S_i)\}$.  We order finite multi sets by arranging the integers in monotonically non-increasing order and compare the ordered multi-sets lexicographically.  

Let $\beta \epsilon \mathbb{Q}$.  Define the \emph{$\beta$-width, $w(M,\beta)$,  of $M$} to be the minimal width over all $\beta$-sloped decompositions, using the above ordering of multi-sets.  The $\infty$-width $w(M,\infty)$ is the minimal width over all meridional surface decompositions.  This notion is similar to both the concept of thin position of the pair $(M^3, c)$, where $c$ is a 1-sub-manifold of $M$ of Hayashi and Shimokawa \cite{HayShiTPoP3M1S} as well as of Tomova \cite{TomTPfKi3M} in the case that $c$ is a closed curve.   Define the \emph{empty-width $w(M,\emptyset)$ of $M$} to be the minimal width over all closed surface decompositions of $M$, which a similar notion to Scharlemann and Thompson's thin position for 3-manifolds \cite{SchaThoTPf3M}, but with a different measure of complexity.      


Define the \emph{width $w(M)$ of M} to be the minimal width of $w(M,\alpha)$ over all $\alpha \epsilon\{\mathbb{Q} \cup \infty \cup \emptyset \}$.  We call any decomposition $\{F_i,S_i\}$ of $M$ \emph{thin} if it realizes the width of $M$ and denote it $thin(M)$.   For $\alpha \epsilon\{\mathbb{Q} \cup \infty \cup \emptyset \}$ we denote any $\alpha$-sloped decomposition realizing $w(M, \A)$, $thin(M, \A)$.

By the above measure of complexity, the complexity of a surface goes down when the surface is either compressed or boundary compressed, which is clearly desirable: a decomposition goes down in width if the decomposing surfaces are obviously simplified, see Figure \ref{figure:compressing}.

  \begin{figure}[htbp]
\begin{center}
\includegraphics[width=3in]{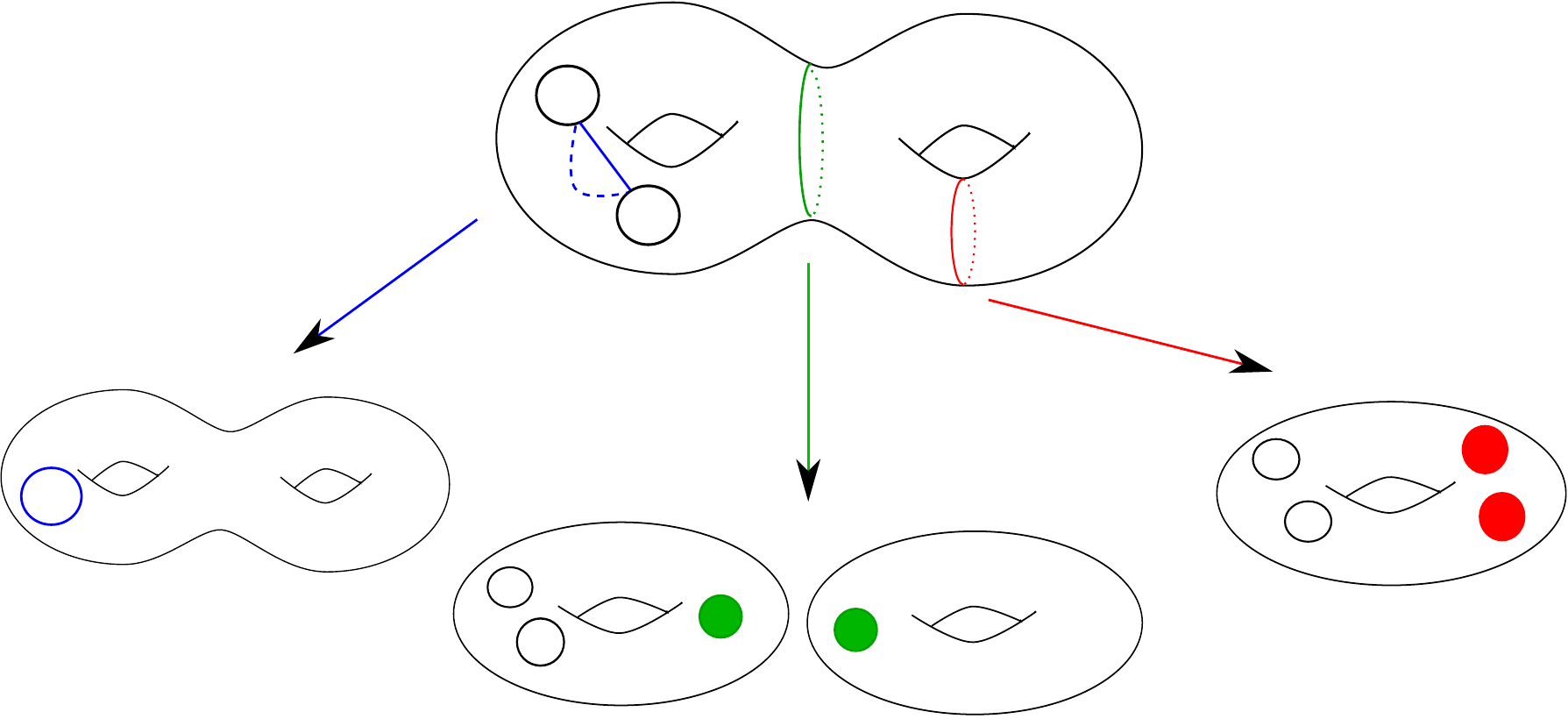}
\caption{Surface complexity goes down with either compression or boundary compression.}
\label{figure:compressing}
\end{center}
\end{figure}

%

Given a knot in $K$ in $S^3$, it is always possible to position $K$ in relation to a Heegaard sphere $S$ for $S^3=B_1 \cup_S B_2$ in such a way that $K\cap B_i$ for $i=1,2$ is a collection of $n$ arcs  which are boundary parallel in $B_i$.  This presentation of a knot is called \emph{a bridge position}.  


This idea can be generalized to knots in other manifolds, positioning $K \subset M^3$ in relation to a Heegaard surface.  For $n \geq 1$, we will say that $K\subset M^3$ is $(t,n)$ if $K$ can be put in $n$-bridge position with respect to a  genus $t$ Heegaard surface $S$. We will say that $K$ is $(t,0)$ if $K$ can be isotoped into $S$. If $K$ is $(t,n)$ for some $n$ then $K$ is $(t, m)$ for every $m \geq n$. Thus we are concerned with the smallest $n$ such that $K$ is $(t, n)$.

In \cite{JohThoTN1KWaN1n} it was shown that there are knots $K\subset S^3$ which are tunnel number one, which are not (1,1).  In the language of $\alpha$ sloped decompositions this means that there are knots whose exteriors have a closed genus two Heegaard surface, but do not have a  twice punctured genus 1 meridional Heegaard surface.  Any knot which has a twice punctured genus 1 meridional Heegaard surface, $S$,  does have a closed genus two Heegaard surface realized by tubing $S$ along $K$ , see Figure \ref{figure:complexitytube}.  Acknowledging this, we consider the (1,1) decomposition, or the twice punctured genus 1 Heegaard surface, to be a simpler decomposition than a tunnel number 1, or closed genus 2, decomposition.  This is reflected in the measure of complexity of surfaces.  In general, the complexity of a surface goes up when the surface is tubed along a boundary component.


  \begin{figure}[htbp]
\begin{center}
\includegraphics[width=3in]{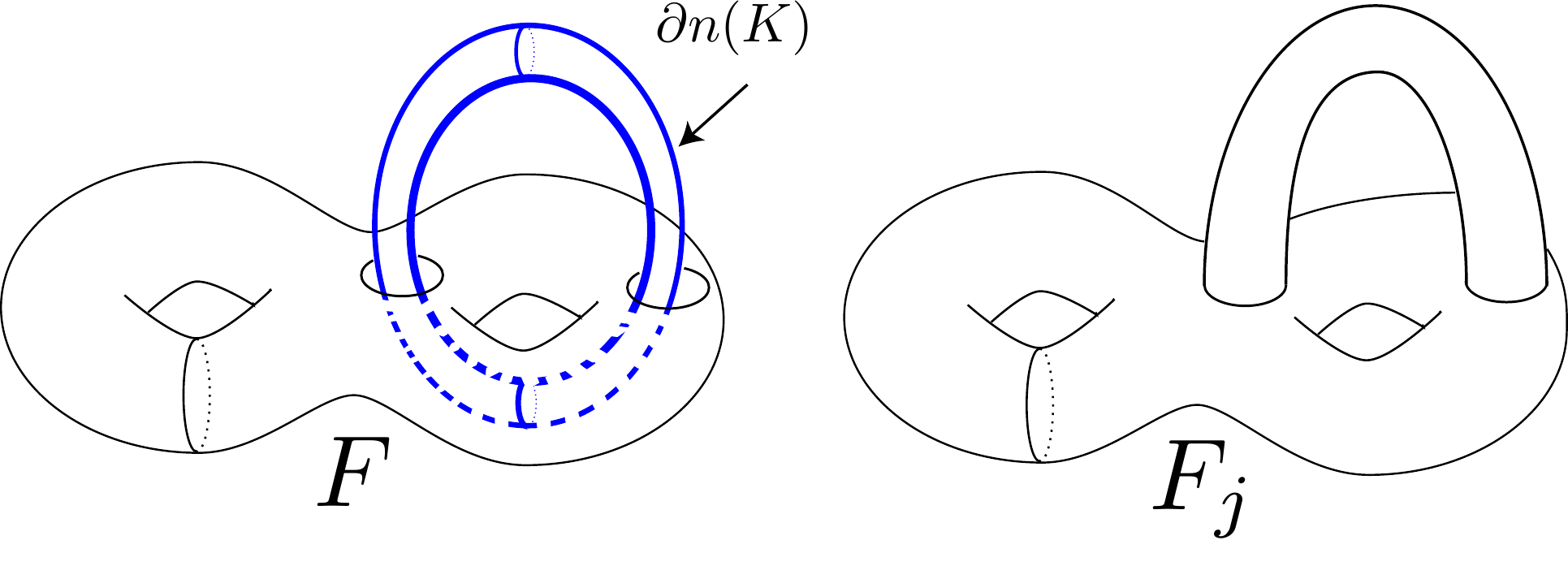}
\caption{Surface complexity goes up when the surface is tubed along a boundary component of $M$.}
\label{figure:complexitytube}
\end{center}
\end{figure}

Let $F$ be a surface with boundary.  Let $|\partial F|$ to be the number of boundary components of $F$.  It follows from our previous observations about boundary compression bodies, see section \ref{subsec: BoundaryCompressionBodies}:

\begin{lemma}
\label{lemma:boundary+lessboundary}
In a boundary compression body $C$, $c(\partial_- C) \leq c(\partial_+ C)$ and $|\partial(\partial_- C)| \leq |\partial(\partial_+ C)|$.  
\end{lemma}

A boundary compression body in which these values are equal is simply $\partial_-C \times I$ and is called a \emph{trivial} boundary compression body.

Notice that we do not require all surfaces in an $\A$-sloped decomposition to have boundary, only that there must be at least one $\alpha$-sloped bead in the handle and bead decomposition.  There are however restrictions on when a decomposing surface can be closed.  

\begin{lemma}
\label{lemma:thinclosed}
Let $M$ be a manifold with torus boundary component $T$ and let $\{F_i, S_i\}$, $i=1...n$, be an $\A$-sloped decomposition of $M$ along $T$.  If some $F_j$ is a closed surface then either all of $\{F_k, S_k\}$, $1\leq k \leq j$ are closed or all of $\{F_k, S_k\}$, $j \leq k \leq n$ are closed.

\begin{proof}

Since $F_i$ is a closed separating surface, and $T$ is connected, $T$ is entirely contained in  one component of $M | F_i= M_1 \cup M_2$, say $M_1$.  Any decomposition of $M_1$ must be a closed decomposition.  Either $\{F_k, S_k\}$, $1\leq k \leq j$ or $\{F_k, S_k\}$, $j \leq k \leq n$ is a decomposition of $M_1$, and thus all the surfaces are closed.  
\end{proof}
\end{lemma}

In order to examine ways in which decomposing surface can be simplified, we recall and generalize some classical notions of Heegaard splittings.  

\begin{definition}We call an $\alpha$-sloped Heegaard splitting $(C_1, C_2, S)$ \emph{weakly reducible} if there is a pair of compressing disks $D_i \subset C_i$, i=1,2, such that $D_1 \cap D_2 = \emptyset$.  

An $\alpha$-sloped Heegaard splitting which is not weakly reducible is \emph{strongly irreducible}, and any compressing disk $D_1 \subset C_1$ intersects every compressing disk $D_2 \subset C_2$.  
\end{definition}

%

\begin{definition}We call an $\alpha$-sloped Heegaard splitting $(C_1, C_2, S)$ \emph{boundary weakly reducible} if there is a pair of disks $D_1, D_2 \subset C_1, C_2$ respectively, each of which is either a $\partial_0$-compressing disk or compressing disk for $C_i, i=1,2$ with $D_1$ and $D_2$ disjoint.  

An $\alpha$-sloped Heegaard splitting which is not boundary weakly reducible is \emph{strongly boundary irreducible}, and any compressing disk or $\partial_0$-compressing disk $D_1 \subset C_1$ intersects every compressing disk or $\partial_0$-compressing disk  $D_2 \subset C_2$.  

\end{definition}

Weakly reducible splittings are weakly boundary reducible, and strongly boundary irreducible splittings are strongly irreducible.  

Note that $\A$-stabilization of  non-trivial Heegaard splittings are weakly reducible.  Also note that the disk $A_{\alpha} - n(\gamma)=D_2$ and a $\partial_0$-compressing disk for $b$, $D_1$ intersect in a single point.  

The process of $\alpha$-stabilization guarantees the existence of multiple non-isotopic $\alpha$-sloped Heegaard splittings of a given manifold, but the non-uniqueness of $\A$-sloped decompositions occurs in a non-trivial way as well.  

\begin{theorem}It is possible to have two non-isotopic $\alpha$-sloped generalized Heegaard splittings of the same width.  

\begin{proof} In \cite{GunKWDPPaPSR} Guntel constructs pairs of twisted torus knots $K_1$, $K_2\subset S^3$ each of which sit in a genus two Heegaard surface $S$ for $S^3=H_1 \cup_S H_2$ with the same surface slope $\alpha$, but which are not isotopic in the surface.  Furthermore they are the same knot $K \subset S^3$.    Two non-isotopic $\alpha$-sloped generalized Heegaard splittings for $K$ can be constructed as follows:  In both cases, build a core of $H_1$ with a 0-handle and two 2-handles, and attach a $\alpha$-sloped 0-bead, attached to the core by a 1-handle $\mathcal{H}$.  Attach a 2-handle along $\partial (A_\alpha - \mathcal{H})$ for a choice of an $\A$-sloped spanning annulus $A_\alpha$ with the property that $A_\alpha - \mathcal{H}$ is a disk.  Call the resulting submanifold $M_i$.  The surface $\partial M_i -\partial M$ is isotopic to $\overline{S - n(K_i)}$.  In each case, attach a 1-handle, followed by a collection of 2-handles, capping-beads, and 3-handles.  Each splitting has a single thin surface, $\overline{S-n(K_i)}$. Since these surfaces are not isotopic, neither are the splittings.  
\end{proof}

\end{theorem}


\section{Properties of Thin Decompositions}
\label{section:PropertiesofThinDecompositions}

In \cite{SchaThoTPf3M}, Scharlemann and Thompson showed that in a thin decomposition, the thick and thin surfaces, as well as the sub-manifolds resulting from cutting up a manifold along them, have nice properties.  If we consider ways in which it is possible to thin an $\A$-sloped decomposition, we see that similar properties are true in this context.  

If we are able to completely replace some collection of handles or beads with a different kind of handle or bead which results in simpler surfaces, then it is possible to thin a decomposition.  In that vein, following Rule 1 and Rule 2 of \cite{SchaThoTPf3M}, we see that:
\begin{lemma}
\label{lemma:thinsphereessential}
In a thin $\alpha$-sloped decomposition any sphere component of any $F_i$ is essential.  
\end{lemma}

\begin{lemma}
\label{lemma:activecomp}
In a thin $\A$-sloped decomposition each component of $F_{i-1}$ either persists into $F_i$, or has handles from both $\mathcal{H}^1$ or $b^0$ and $\mathcal{H}_2$ or $b^2$.  

\end{lemma}

$\alpha$-stabilized surfaces can be simplified to obtain lower complexity decomposing surfaces, removing a bead entirely  from a decomposition, and so we conclude:
\begin{lemma}
\label{lemma:dishonestthinisstablized}
If $thin(M, \alpha)$ is $\alpha$-stabilized then it is obtained from $thin(M, \emptyset)$ via a single $\alpha$ stabilization.

\end{lemma}

\begin{theorem}
\label{theorem:SiAreStronglyIrred}

In a thin $\alpha$-sloped decomposition each $S_i$ is either
\begin{enumerate}
 \item a strongly  boundary irreducible $\alpha$-Heegaard splittings of $W_i=C_i^1 \cup C_i^2$ or 
 \item $\alpha$-stabilized.   
 
In the second case $S_1$ is an $\A$-stabilized Heegaard surface for $W_1$ and all other $S_j$ are strongly irreducible closed Heegaard surfaces for $W_j$.

\end{enumerate}
\begin{proof}

Let $S_i$ be a thick surface in a thin $\alpha$-sloped decomposition of $M$ along $T$.  Suppose there was a compressing or $\partial_0$-compressing disk $D_1\subset C_1$ which was disjoint from a compressing or $\partial_0$-compressing disk $D_2\subset C_2$.  

If $D_1$ and $D_2$ are both $\partial_0$-compressing disks then by Lemma \ref{lemma:boundaryplusboundarycomp} they are the co-cores of the 0-bead $b_1$ and the 2-bead $b_2$ respectively.  Let $\partial(\partial_0 b_2)=c_1 \cup c_2$. If $\partial(\partial_0 b_1) \neq \partial(\partial_0 b_2)$ then after possibly renaming, $c_1= [\partial(\partial_0 b_1) \cap \partial(\partial_0 b_2)]$ and $c_2$ is another boundary component of $S_i$.  



  Let $S_i^j$, $j=1,2$,  be the result of compressing $S_i$ into $C_j$ along $D_j$.      On the side not containing $D_j$, $S_i^j$ bounds a boundary compression $C'_j=C_j-b_j$.  The curve $c_1$ is not a boundary component of $S_i^1$, but $c_2$ is.    Since $b_1 \cup b_2$ is the neighborhood of some annulus $A$ of $T \setminus (\partial S'_i \cup \partial F_2)$ is is possible to isotope $S_i^1$ by pulling $c_2$ across $A$ to $S_i^2$.  Now $S'_i=S_i^1=S_i^2$ is a new Heegaard surface for $W_i$ with $\{c(S'_i)\}<\{c(S'_i)\}$, as so is a thinner decomposition, contradicting our assumption.

If $\partial(\partial_0 b_1)=\partial(\partial_0 b_2)$ then $b_1 \cup b_2= n(T)$ and $\partial S_i=c_1 \cup c_2$.  As above let $S_i^1$ be the closed surface which is the result of compressing $S_i$ into $C_1$ along $D_1$.  Note that the genus of $S_i^1$ is the same as the genus of $S_i$ and it is closed.   On the side not containing $D_1$, $S_i^1$ bounds a compression $C'_1=C_j-b_j$.  The  distinguished boundary component $T$ of $M$ is contained in the boundary of the submanifold $C'_2=C_2 \cup b_1$ of $M$ on the side of $S_i^1$ containing $D_j $.

Let $\beta_j$, $j=1,2$  be the arc of $D_j$ contained in $S_i$. Then $\beta_1 \cup \beta_2$ is a simple closed curve in $S_i$.    It is possible to isotope $C'_1$ so that it is incident to $b_1$ via $\partial_+ b_1 - n(\beta_1)$.  Thus $C_2'=C_2 \cup_{\beta_1}b_1$ with the boundary components of $C_2$ and $b_1$ corresponding to $c_1$ and $c_2$ identified correspondingly.  Now $C'_2$ is a compression body and $C'_1 \cup_{S_i^1} C'_2$ is a closed Heegaard splitting of $W_i$ of the same genus as $S_i$.  Note that it is also possible to obtain $S^1_i$ by compressing $S_i$ along the separating disk for $b_1$, and so by lemma \ref{lemma:alphastabcomptoclosed} $S_i$ was alpha stabilized.  

If $\partial_0 b_1 \cap \partial_0 b_2 = \emptyset$ or either $D_1$ or $D_2$ is a compressing disk, remove a neighborhood $n(D_1)$ from $C_1$ converting it into a boundary compression body $C_1'$ with either one fewer 0-bead or one fewer 1-handle.  

If $D_2$ is a compressing disk, attach a 2-handle with core $D_2$  to $C_i$. If it is a $\partial_0$-compressing disk, attach a 2-bead with co-core $D_2$ to $C_1'$, then attach a 0-bead or 1-handle which is dual to $D_1$, followed by the rest of the two handles and 2-beads of $C_2$.   Now we have replaced $S_i$ with a new decomposition of $W_i$.  The width of the original decomposition was $\{c(S_i)\}$, and the width of the  new decomposition is $\{c(S_i^1),c(S_i^2)\}$, where $S_i^1$ is the result of (boundary) compressing $S_i$ along $D_1$, and $S_i^2$ is the result of (boundary) compressing $S_i$ along $D_2$, and so $c(S_i^1),c(S_i^2) < c(S_i)$ and the new decomposition is thinner.

If the new decomposition has $\alpha$-sloped surfaces then we have contradicted our assumption that we began with a thin $\alpha$-sloped decomposition.  Thus, the resulting surfaces must be closed.  In this case, $thin(M, \alpha)$ must be $\A$-stabilized and by Lemma \ref{lemma:dishonestthinisstablized},  $S_1$ an $\A$-stabilized Heegaard surface for $W_1$.  
\end{proof}

\end{theorem}
The following lemma is a generalization of Haken's Lemma and Lemma 1.1 of \cite{CasGorRHS} and \cite{BonOtaSdHdEL}.  
\begin{lemma}
\label{lemma:Haken'sLemma}

Let $W=C_1 \cup_S C_2$ be a boundary Heegaard splitting of $W$.  Let $\mathcal{S}$ be a disjoint union of properly embedded essential 2-spheres, and compressing- and $\partial_0$-compressing disks for $\partial_- C$.  Then there exists a disjoint union of essential 2-spheres and disks $\mathcal{S}^*$ in $W$ such that:

\begin{enumerate}

\item $\mathcal{S}^*$ is obtained from $\mathcal{S}$ by ambient 1-surgery and isotopy
\item each component of $\mathcal{S}^*$ meets $S$ in a single circle
\item there exist  cut systems $\mathcal{D}_1$, $\mathcal{D}_2$ for $C_1$, $C_2$ respectively such that $\mathcal{D}_1 \cap \mathcal{S}^*=\mathcal{D}_1 \cap \mathcal{S}^*=\emptyset$.  
\end{enumerate}

\end{lemma}

\begin{proof}

Our previous observations show that it is not possible for $\partial_-C$ to have a $\partial_0$-compressing disk, so $\mathcal{S}$ contains essential spheres and compressing disks for $\partial_-C$ .  

Let $D$ be an essential 2-sphere, or a compressing disk element of $\mathcal{S}$, which has the fewest number of intersections with $S$ of all such spheres or disks in its isotopy class.    Curves of intersection between $D$ and $S$ which are inessential in $S$ can be ruled out by the fact the $C_i$ is irreducible and that $D \cap S$ is minimized.  At least one component of $D|S \cap \partial_- C_i= \emptyset$, $i=1,2$.  Call this component $D^*\subset C_1$, say.  If $D^*$ is a disk, then $D \cup S$ is a single essential curve on $S$.

If $D^*$ contains a component which is not a disk, using a hierarchy of surfaces and a series of boundary compressions, push strips of $D^*$ across $S$ into $C_2$ until $D^*\cup C_1$ is a collection of disks, see Jaco's \cite{JacLo3MT} account of the proof of Haken's lemma for more details.  Call the image of $D^*$ after this series of isotopies $D'$.  By lemma 11.9 of \cite{JacLo3MT} the number of intersections between $D'$ and $S$ is fewer than the number of intersections between $D$ and $S$, a contradiction.  Thus $D^*$ is a disk and $D' \cap S$ is a single circle.

To find $\mathcal{D}_i$ cut $C_i$ along the collections $\{\mathcal{S}_i\} $ of disk components of  $\mathcal{S}^* \cap C_i$, obtaining a collection of boundary compression bodies $\{C_i^j\}$, $j=1,2,..,n$.  Let $\{\mathcal{S^*}_i\}$ be a collection of parallel copies of the disks $\{\mathcal{S}_i\} $ and let $\mathcal{D}_i^j$ be a cut system for this collection of boundary compression bodies $\{C_i^j\}$.     Then $[\cup_j\mathcal{D}_i^j ]\cup \{\mathcal{S^*}_i\}$ contains a cut system which is disjoint from $\mathcal{S}^*$.  Call this system $\mathcal{D}_i$.  
\end{proof}
 
\begin{lemma}
If $\partial_- W_i = F_{i-1} \cup F_i$ is compressible or $\partial_0$-compressible in $W_i$ then $S_i$ is a weakly reducible Heegaard surface for $W_i$.  

\begin{proof}

As above, it is impossible for $F_i$ to be $\partial_0$-compressible.  Let $D$ be a compressing disk for $F_i$, say.  By Lemma \ref{lemma:Haken'sLemma} $D$ is isotopic to a disk which intersects $S_i$ in a single essential circle $c$, and is disjoint from cut systems $\mathcal{D}_1$, $\mathcal{D}_2$ for $C_1$, $C_2$ respectively.  The curve $c$ cuts $D$ into an annulus $A\subset C_1$, and a compressing disk $D'$ for $S_i$.  Since $D' \cap \mathcal{D}_2 = \emptyset$, $S_i$ is a weakly reducible Heegaard surface.  
\end{proof}

\end{lemma}

We see from \cite{SchaThoTPf3M} Rule 5 and observations in Lemma \ref{lemma:negboundaryincomp}: 

\begin{theorem}
\label{theorem:ThinInc}
 In a thin $\A$-sloped decomposition, each component of each $F_i$ is incompressible and boundary incompressible in $M$.

\end{theorem}

\begin{corollary} If $\alpha$-sloped thin position is not an $\alpha$-sloped Heegaard splitting, then $M$ contains an essential, $\alpha$-sloped or closed surface.  
\end{corollary}

A separating surface $S$ is called \emph{boundary weakly incompressible} if any two compressing or $\partial_0$-compressing disks for $S$ on opposite sides of $S$ (boundary) intersect.  

Finally we see from Rule 6 of \cite{SchaThoTPf3M}

\begin{lemma}
In a thin $\A$-sloped decomposition, each $S_i$ is boundary weakly incompressible in $M$. 
\end{lemma}

\subsection{Honest Slopes}
\label{subsection:HonestSlopes}
  
While the above construction shows that it is always possible to decompose a manifold $M$ with torus boundary using surfaces with any slope, for some slopes on $\partial M$, $\alpha$-stabilization of a closed decomposition is the only way to obtain an $\alpha$-sloped decomposition.  We would like to be able to pick out slopes for which this is true.  In general we would like to consider splittings which are not stabilized or $\alpha$-stabilized.

We call an $\alpha$-sloped decomposition $\{F_i, S_i\}$ of $M$ $ \emph{dishonest}$ if it is $\A$-stabilized, and \emph{honest} otherwise.   We call a slope $\alpha \subset \partial M$ $\emph{honest}$ if $M$ has an $\alpha$-sloped decomposition which is honest.  

If $M$ contains an $\alpha$-sloped two sided essential surface $S$, $\alpha$ is an honest slope.  We can see this by cutting $M$ along $S$ and considering a handle and bead decomposition of each component, $M_1$ and $M_2$.  Gluing $M_1$ and $M_2$ back together along $S$ gives an $\alpha$-sloped decomposition of $M$.  If this decomposition is $\alpha$-stabilized, $\A$-destabilize as much as possible.  Since $S$ is incompressible, it is not an $\alpha$-stabilized surface, and thus the $\A$-destabilized decomposition still has slope $\alpha$, hence $\alpha$ is an honest slope.  

$\alpha$-sloped thin surfaces are essential, but not all honest decompositions have these essential thin surfaces.  It is known that the set of slopes with essential surfaces is finite \cite{HatOtBCoIS}, but it is unknown if the set of honest slopes is also finite.  Putting this fact together with Theorem \ref{theorem:ThinInc} we see

\begin{theorem}
\label{theorem:finitelymany}
For all but finitely many $\alpha$, $thin(M,\alpha)$ is an $\A$-Heegaard splitting or a dishonest decomposition.

\end{theorem}

    \section{Induced Splittings and Dehn Filled Manifolds}
    \label{sec:InducedHSplittings}

  Both surfaces with boundary and boundary compression bodies in $M$ are ``filled in" in a nice way by Dehn filling.     As usual, let $M$ be a compact, orientable manifold with torus boundary components, each of which is parametrized by elements of $\{\mathbb{Q} \cup \infty\}$.  

Select a torus boundary component $T$ of $M$. \emph{The manifold which results from $\alpha$-sloped Dehn filling $M$ along T, $M_T(\A)$} is $M \cup_T S^1 \times D^2$, where a meridian of $S^1 \times D^2$ is identified to a circle on $T$ with slope $\alpha$.  
 If $M$ only has one boundary component or it is clear which component is being Dehn filled, the Dehn filling will be denoted $M(\A)$.  If $K\subset S^3$ we denote $\alpha$-sloped Dehn surgery on $K$ ($\A$-sloped Dehn filling on the exterior of $K$) by $K(\A)$. 

 Denote the surface in $M_T(\A)$ resulting from ``capping off" the boundary components  ${\partial_i F}$ of $F$ by $\hat{F}$.   More precisely, $\hat{F}= F \cup_{\partial F} \{D_i\}$.  Note that if $M$ only has one boundary component then $\hat{F}_T$ is a closed surface.  In either case, the genus of $F$ equals the genus of $\hat{F}_T$.   
%

If $C\subset M$ is a boundary compression body with  with boundary slope $\A$ on $T$ and $\partial_0$-components ${{\partial_0} C}$ on $T$, then the Dehn filling $M_T(\A)$ results in attaching a two handle along every component of ${{\partial_0} C}$.  Denote the compact sub-manifold of $M(\A)$ resulting from ``capping off" ${{\partial_0} C}$ by $\hat{C}$.  The genus of both $\partial_-C$ and $\partial_+ C$ remain unchanged, and the number of boundary components of $\partial_-C$ and $\partial_+ C$ as well as  the number of components of $\partial_0 C$ go down.  


Putting these two facts together we see:

\begin{theorem}
\label{theorem:cappingoff}
 If $\{F_i, S_i\}$ is an $\alpha$-sloped generalized Heegaard splitting of $M$ then $\{\hat{F_i}, \hat{S_i}\}$ is a  generalized Heegaard splitting of $M(\alpha)$.  We call  $\{\hat{F_i}, \hat{S_i}\}$  the \emph{induced generalized Heegaard splitting} of $M(\A)$.  For completeness, the induced splitting coming from a $\emptyset$-sloped decomposition $\{F_i, S_i\}$ of $M$ is again $\{F_i, S_i\}$.
 
 \end{theorem}

Exploiting this we see that for knots in $S^3$ there are restrictions on what types of thin decompositions are possible.

\begin{theorem}
\label{theorem:planar}
Let $K\subset S^3$.  If $\A$-sloped thin position for $K$ is realized by a single planar Heegaard surface, then $\alpha=\infty$.  

\begin{proof}Say $\alpha \neq \infty$.  Then $\alpha$-sloped Dehn filling on $K$ results in  $\hat{S}$, a genus 0 Heegaard surface for $K(\alpha)$.  Thus $K(\alpha)$ must be $S^3$ and was the result of non-trivial surgery on $K$, a contradiction \cite{GorLueKADBTC}.  
\end{proof}
\end{theorem}

It is natural to ask under what circumstances  Heegaard genus drops under Dehn filling.  For results on the subject see for example \cite{MorSedTHSoDFM}, \cite{RieSedPoHSUDF}. The same question is of interest in the context of $\alpha$-sloped Heegaard genus.  One would hope that performing $\A$-Dehn filling on a thin decomposition of $M$ would result in a thin, or at least locally thin decomposition of $M(\A)$.   Examples where this is not the case, however, come from the simplest Dehn filling, the meridional filling.

\begin{example}
Let $K\subset S^3$ be a knot which is $(2,1)$,  $(1,k)$ for $k \geq 3$ and $(0, s)$ for $s \geq 5$.  Then $w(K, \infty)=\{7\}$ and $thin(K, \infty)=\{S\}$, the twice punctured genus 2 Heegaard surface for $S^3$ which realizes $K$ as $(2,1)$.  Clearly  $\hat{S}\subset K(\infty)=S^3$ is not a thin Heegaard surface, nor is it a strongly irreducible Heegaard surface.  
\end{example}

We can also use this fact to detect honest slopes.  Any slope for which there is a drop in Heegaard genus under Dehn filling is honest.

\begin{theorem}
\label{theorem:dropgenushonest}

Let $\alpha \epsilon \mathbb{Q}$ and let $M$ be a compact orientable manifold  with torus boundary. If $g(M(\alpha))<g(M)$ then $\alpha$ is an honest slope.  

\begin{proof}
Let  $K'$ be as above, $g=g(M)$ and let $S$ be a minimal genus Heegaard surface for $M(\alpha)$.  Put $K'$ in minimal bridge position with respect to $S$.   $S-n(K')=S' \subset \overline{M(\A)-n(K')}=M$ is an $\alpha$-sloped Heegaard surface for $M$.  If $S'$ were alpha stabilized then by Lemma \ref{lemma:alphastabcomptoclosed} there would be a compressing disk $D$ for $S'$ such that $S' | D$ was a closed Heegaard surface for $M$ with the same genus as $S'$.  But since $g(M(\alpha))<g(M)$, $g(S)<g(M)$, a contradiction.  Thus we have exhibited an honest $\alpha$-sloped decomposition for $M$, and $\A$ is an honest slope.  
\end{proof}
\end{theorem}

\begin{corollary}For every $K\subset S^3$, $\infty$ is an honest slope.  

\begin{proof}

Since $M(\infty)=S^3$, which has genus 0, by Theorem \ref{theorem:dropgenushonest} $\infty$ is an honest slope.  
\end{proof}
\end{corollary}

For the following we will need a theorem of Culler, Gordon, Luecke and Shalen \cite{CulGorLueShaDSoK}.  We rephrase and restate the relevant parts for completeness.

\begin{theorem}[Theorem 2.0.2 of \cite{CulGorLueShaDSoK}]
Suppose that $dim H_1(M, \mathbb{Q})>1$.  If $M$ contains an $\alpha$-sloped essential surface then either
\begin{enumerate}
\item $M(\alpha)$ is a Haken manifold
\item $M(\A)$ is a connected sum of two lens spaces
\item $M$ contains a closed incompressible surface
\item $M$ fibers over $S^1$ with fiber a planar surface having boundary slope $\alpha$.  
\end{enumerate}

\end{theorem}


%
%
%
%
%
%
%
%

We are now able to prove Theorem \ref{theorem:MainTheorem}.  

\begin{restatemain}

Let $K \subset S^3$ be a non-trivial knot and let $\alpha \epsilon \{\mathbb{Q} \}$.  Let $S^3-K$ be in $\alpha$-sloped thin position, and let $K(\alpha)$ be the result of $\alpha$-sloped Dehn surgery on $K$.  Then either 
\begin{enumerate}

\item{$\A$-sloped thin position for $S^3-K$ is an $\A$-sloped Heegaard surface }

\item{There exists a closed essential surface in the exterior of $K$}
\label{item:closed}

\item{$K(\alpha)$ is Haken}
\label{item:Haken}

\item{$K(\alpha)$ is a connected sum of two lens spaces}
\label{item:lensspace}

\end{enumerate}
\end{restatemain}

\begin{proof}

If $\A$-thin position for $K$ has more than one thick surface, there is at least one thin surface $F$.  By Theorem \ref{theorem:ThinInc} $F$ is essential. If $F$ is closed then $K$ has a closed essential surface and case \ref{item:closed} happens.  If $F$ has boundary then $K$ has an essential $\A$-sloped surface.  By the proof of  \cite{CulGorLueShaDSoK} Theorem 2.0.2 either one of options  \ref{item:Haken} or \ref{item:lensspace} happens or $K$ fibers over $S^1$.  In the latter case $K$ is the unknot, a contradiction.  
\end{proof}


      \section{Restrictions on Width}
    \label{sec:RestrictionsOnWidth}

This relationship between $\A$-sloped Heegaard splittings and Heegaard splitting of Dehn filled manifolds allows us to make restrictions on the possible widths of knots in $S^3$.  Let $\alpha \epsilon \{\mathbb{Q} \cup \infty\}$.  

 \begin{theorem} 
  \label{theorem:1s}
Let $K \subset S^3$.   If $w(K, \alpha ) = \{1\}$ then $K$ is the unknot.  If $w(K, \alpha )$ contains a $1$ then $K$ is the unknot.

 \begin{proof}
Say $w(K, \alpha ) = \{1\}$, and call the boundary Heegaard surface $S$.  $c(S)=1$ implies that $S$ is a annulus, and thus by Theorem \ref{theorem:planar}, $\alpha=\infty$, $K$ is in standard thin position for $K\subset S^3$, and so $K$ is the unknot.  

If $w(K, \alpha )$ contains a $1$ but is not equal to $\{1\}$ then a thin decomposition of $K$ has more than one thick level, and thus  must have thin levels.  By Lemma \ref{lemma:activecomp} one of these thin levels must have complexity less than 1, i.e. 0.  But since thin levels are essential, this implies that $K$ contains an essential sphere, which is not possible.  Thus if $w(K, \alpha )$ contains a $1$ it must be equal to $\{1\}$, and $K$ is the unknot.  
\end{proof}
\end{theorem}

 \begin{theorem}
 \label{theorem:3s}
 Let $K \subset S^3$ be a non-trivial knot.   If $w(K, \alpha) = \{3\}$ then $\alpha=\infty$ and $K$ is a two bridge knot.   If $w(K, \alpha)$ contains a 3, but is not equal to $\{3\}$ or is less than $\{3\}$, then either the exterior of $K$ contains an essential torus or the exterior of  $K$ contains an essential annulus and $K(\alpha)$ is the connect sum of two lens spaces.  
 
 \begin{proof}  If $w(K, \alpha ) = \{3\}$, the boundary Heegaard surface $S$ is a 4-punctured sphere, and thus by Theorem \ref{theorem:planar} $\alpha=\infty$, and $K$ is in standard bridge position, and is a 2-bridge knot.  
 
 If $w(K, \alpha)$ contains a 3, but is not equal to $\{3\}$, then there is a thin surface, $F_i$ of complexity lower than 3, i.e. 0, 1, or 2.  If there is a thin surface of complexity 0, then $K$ contains an essential sphere, a contradiction.  If there is a thin surface of complexity 1, i.e. an annulus,  by the proof of Theorem 2.0.2 of \cite{CulGorLueShaDSoK}, either $K$ contains an essential torus disjoint from $F_i$, or $K(\alpha)$ is the connect sum of two lens spaces.    If there is a thin surface of complexity 2, it is an essential torus, or a pair of annuli .  In the latter case, by proposition 2.3.1 in \cite{CulGorLueShaDSoK} $K$ fibers over $S^1$, and thus is the unknot, a contradiction.

 If $w(K, \alpha) < \{3\}$ then, by Theorem \ref{theorem:1s} $w(K, \alpha)$ is a n-tuple of 2s.  If $n=1$ then $K$ has a genus 1 Heegaard splitting, a contradiction.  If $n>1$ then there is a thin surface $F_i$ of complexity lower than 2, i.e. 0 or 1.  By the arguments above, either $K$ contains an essential torus, or $K(\alpha)$ is the connect sum of two lens spaces. 
 \end{proof}

\end{theorem}

    \section{Complexity Bounds}
        \label{sec:ComplexityBounds}
        
        In section \ref{sec:RestrictionsOnWidth} we saw that there are some universal restrictions on the widths of knots in $S^3$.  Here we explore the range of possible widths for a fixed manifold, over all $\A \epsilon \{\mathbb{Q} \cup \infty \cup \emptyset\}$
    
    To begin we establish some notation.  Let $\{c_i\}_{1\leq i\leq k}$ and $\{n_i\}_{1\leq i\leq k}$ be ordered multi sets.  Define $\{c_i\} + \{n_i\}$ to be the multi-set $\{c_i+n_i\}$.  Define $\{c_i\} +_i m$ to be the multi-set $\{c_1, c_2, c_3,..., c_i+m, .....c_k\}$.  

Define $ n \times \{c_i\} = \{nc_i\}_{1\leq i\leq k}$ 

Define $\lceil \{c_i\}\rceil=\{\lceil c_i\rceil \}$, where $\lceil c_i \rceil$ is the least integer greater than $c_i$.

As usual, let $M$ be a manifold with a designated a torus boundary component $T$.

\begin{theorem} 
\label{theorem:Hgenus}
If the Heegaard genus of $M$ is $g$, then $w(M, \alpha) \leq \{3g+1\}$ for any $\alpha \epsilon \{\mathbb{Q} \cup \infty \cup \emptyset\}$.  

\begin{proof}

A decomposition of this width exists.  Let $S$ be a genus $g$ Heegaard surface for $K$.  $\alpha$-stabilize $S$ once along $T$ to obtain $S'$, a twice punctured genus $g$ $\A$-sloped Heegaard surface for $M$.  $c(S')=3g+1$.  
\end{proof}

\end{theorem}

More generally, the width of any slope is bounded by a function of the closed width, $w(M, \emptyset)$.  In fact, the closed width is always close to as complicated as any $\alpha$-sloped width can be.  Moreover, the higher the closed width, the wider the range of possible $\alpha$-sloped widths over all $\alpha$.    We now prove:

\begin{restatebounds} 

 For any $\alpha \epsilon \{\mathbb{Q} \cup \infty\}$,   $\lceil \frac{2}{3} \times w(M, \emptyset)\rceil \leq  w(M, \alpha) \leq w(M, \emptyset) +_1 2 $.

\begin{proof}

Let the collection of surfaces $\{F_i, S_i\}$ for $i=1,....,n$ be a thin $\alpha$-sloped decomposition of $M$ along $T$ of width $w(M, \alpha)$.

Starting with $S_n$, tube along $\partial M$ to obtain a closed surface $\check{S_n}$.   Next tube $F_{n-1}$ along $T$.  Continue in this manner until there is a closed decomposition of $M$.  If $c(S_i) = 1-\chi(S_i)+g(S_i)=3g_i + p_i-1$ where $g_i$ is the genus of $S_i$, and $p_i$ is the number of boundary components, then $c(\check{S_i})=1-\chi(\check{S_i})+g(\check{S_i})= 3(g_i+ \frac{p_i}{2})-1= 3g_i+\frac{3}{2}p_i-1 <  \frac{3}{2} c(S_i)$.   Thus $w(K, \emptyset) \leq \frac{3}{2} \times w(M, \alpha)$, see Figure \ref{figure:tubing}

Let the collection of surfaces $\{F_i, S_i\}$ for $i=1,....,n$ be a thin closed decomposition of $M$, so $w(M, \{F_i, S_i\})=w(M, \emptyset)$.  After possibly reversing the order of indices, $T \subset C_1 \subset W_1$.  After $\alpha$-stabilizing $S_1$ once we obtain an $\A$-sloped decomposition $\{F_i', S_i'\}$ where $F_i=F_i'$, $i=1,....,n$ and $S_i=S_i'$ for $i=2,....,n$ and $c(S_1')=c(S_i)+2$.  Thus $w(M, \A) \leq w(M, \{F_i', S_i'\})= w(M, \{F_i, S_i\})+_1 2 $
\end{proof}
\end{restatebounds}

  \begin{figure}[htbp]
\begin{center}
\includegraphics[width=3in]{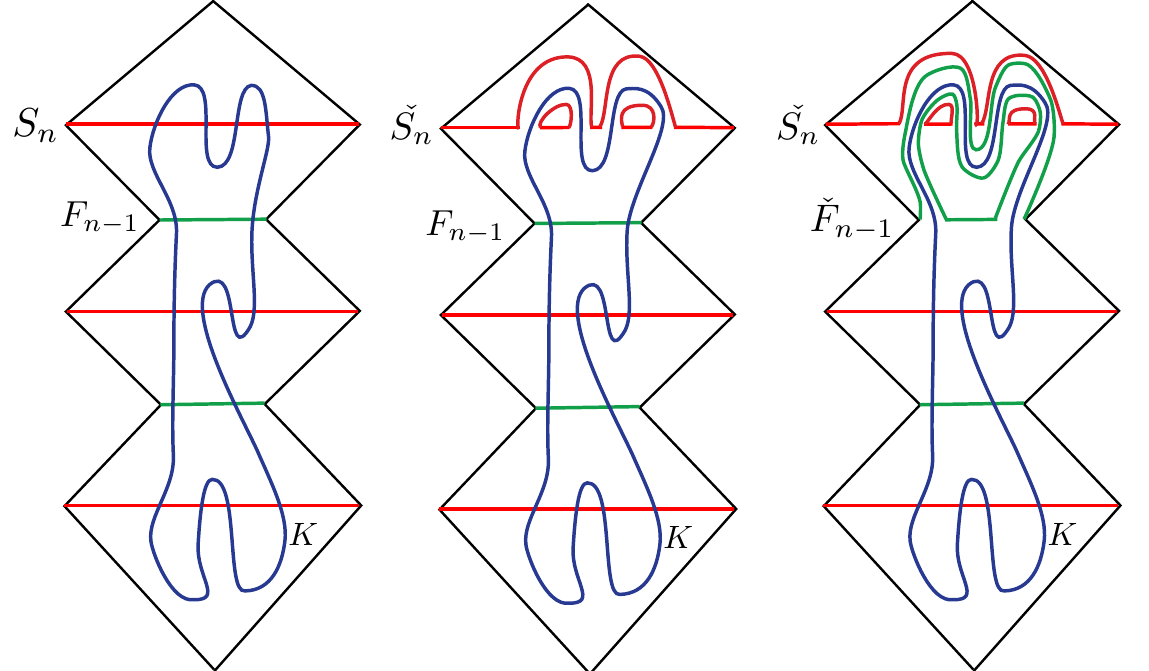}
\caption{An $\A$-sloped surface is tubed along a component of the boundary of $M$}
\label{figure:tubing}
\end{center}
\end{figure}

\begin{remark}
We will see in section  \ref{sec:TorusKnots} that this bound is sharp.  Let $K \subset S^3$ be a torus knot.  In section \ref{sec:TorusKnots} we show that $w(K, \emptyset)=\{5\}$,  $w(K, \alpha)=\{4\}$ for some $\alpha$, and $w(K, \beta)=\{7\}$ for some $\beta$.    Thus $w(K, \beta)=\{7\}=w(K, \emptyset)+ 2$  and $\lceil \frac{2}{3} \times w(K, \emptyset)\rceil = \{\lceil \frac{10}{3}\rceil \}=w(K, \alpha)=\{4\}$.  

\end{remark}

    \section{Torus Knots}
    \label{sec:TorusKnots}
    
    To illustrate the fact that $\alpha$-sloped multiple Heegaard splittings differ depending on $\alpha$ we consider the example of torus knots.  For rational $\alpha$, $\alpha$-sloped decompositions of torus knots exteriors can be grouped into four categories.  In each case the induced Heegaard splitting of the Dehn filled manifold is thin.    
    
Dehn surgeries which yield reducible manifolds and Dehn surgeries which yield lens spaces are the sources of two important conjectures in low dimensional topology: the cabling conjecture and the Berge conjecture.  Not only are Dehn surgeries on torus knots classified, and are well understood\cite{MosESAaTK}, but there exist surgeries yielding lens spaces and others yielding reducible manifolds. Examining thin $\alpha$-sloped decompositions of torus knots across rational $\alpha$ gives a good picture of how these decompositions can be used to see the manifold which results from Dehn surgery.

%
%
%

%

\begin{example}

 Let $K \subset S^3$ be a non-trivial $(p,q)$ torus knot, and let $\frac {r}{s} \epsilon \{\mathbb{Q}\}$ be a slope on the torus boundary $T$ of $M$, the exterior of $K$.

\begin{enumerate} 

\item If $|pqr+s|=1$, then $w(K, \frac {r}{s})=\{4\}$.  In this case $\frac {r}{s}$ is a genus 2 primitive/primitive slope , and $K(\frac {r}{s})$ is a lens space \cite{MosESAaTK}.  

\item If $|pqr+s|=0$, then $w(K, \frac {r}{s})=\{4, 4\}$.  In this case $\frac {r}{s}$ is a cabling annulus slope and, and $K(\frac {r}{s})$ is the connect sum of two lens spaces \cite{MosESAaTK}.  

\item If $|pqr+s|\neq1$ or 0,  then $w(K, \frac {r}{s}) = \{7\}$.  In this case $K(\frac {r}{s})$ is a Seifert fibered space over a sphere with 3 exceptional fibers \cite{MosESAaTK}.

\item $w(K, \emptyset)=\{5\}$.

\end{enumerate}

\end{example}

\begin{proof}
\begin{enumerate}

\item If $|pqr+s|=1$, then $w(K, \frac {r}{s})=\{4\}$.\\
Claim 1: $w(K, \frac {r}{s}) \leq \{4\}$.

A decomposition of this width exists.  First observe that $K$ can sit in  genus two Heegaard surface $S$ for $S^3=H_1 \cup_S H_2$ in a primitive/primitive position with the surface slope  $\frac {r}{s}$.   Begin with a 0-handle and a 1-handle, forming  a core of $H_1$.  Add a $\frac {r}{s}$-sloped 0-bead along $T$, and attach it to the core with a 1-handle to form the boundary compression body $C_1$.   $M - C_1$ is homeomorphic to $C_1$ and thus can be decomposed in the same manner.   Thus $\partial_+ C_1$ is a $\frac {r}{s}$-sloped Heegaard surface which is a twice punctured torus, and has complexity  $\{4\}$.       \\

Claim 2: $w(K, \frac {r}{s}) \geq \{4\}$.  If $w(K, \frac {r}{s}) < \{4\}$, then $w(K, \frac {r}{s})$ must contain 3 or be less than $\{3\}$. If $w(K, \frac {r}{s})= \{3\}$, by Theorem \ref{theorem:3s} $\frac {r}{s}=\infty$, and the so $p,q,r=1$, and $K$ is the unknot, a contradiction.  If $w(K, \frac {r}{s})$ contains a 3, but is not equal to $\{3\}$ or is less then $\{3\}$, then either $K$ has an essential torus or $K(\alpha)$ is the connect sum of two lens spaces.  By \cite{MosESAaTK} $K(\alpha)$ is a lens space.  By \cite{ThuTDGaT1} torus knots do not have essential tori, and so $w(K, \frac {r}{s}) \geq \{4\}$.

Note that this construction works for any Berge knot \cite{BerSKWSYLS}.

\item If $|pqr+s|=0$, then $w(K, \frac {r}{s})=\{4, 4\}$. \\
Claim 1: $w(K, \frac {r}{s}) \leq \{4, 4\}$.

	A decomposition of this width exists.  First observe that $K$ can sit in a  Heegaard torus $S$ for $S^3=H_1 \cup_S H_2$,  as a standardly positioned $(p,q)$ torus knot with surface slope $\frac {r}{s}$.  $M$ is the  cabling annulus of $K$.  Let $A_{\frac{r}{s}}$ be a spanning annulus running between $\partial n(K)$ and a core of $H_1$ with slope $\frac{r}{s}$ on $\partial n(K)$. 

	Begin the decomposition with a 0-handle and 1-handle forming the core of $H_1$. Attach an $\frac {r}{s}$-sloped 0-bead $b$ on $T$ to this core via a 1-handle $\mathcal{H}$ to form the boundary compression body $C_1^1$ so that $b \cap A_{\frac{r}{s}}$ is a co-core of $b$ and $A_{\frac{r}{s}} \cap \mathcal{H}$ is a disk.    Attach a 2-handle to $C_1$ along $\partial_+ C_1^1 \cap A_{\frac{r}{s}}$, forming the boundary compression body $C_1^2$.  Note that $\partial_- C_1^2$ is the cabling annulus and $W_1=C_1^1 \cup C_1^2=H_1-n(K)$.   Also note that $W_1$ is homeomorphic to $M - W_1=W_2$.  Thus $W_2=C_2^1 \cup C_2^2=H_2-n(K)$ where $C_1^i$ is homeomorphic to   $C_2^i$, $i=1,2$.  This decomposition has two thick surfaces, $\partial_+ C_1^1$ and $\partial_+ C_2^1$, each of which are twice punctured tori, and so the width of this decomposition is $\{4,4\}$.  \\

Claim 2: $w(K, \frac {r}{s}) \geq \{4, 4\}$.   

If $w(K, \frac {r}{s}) < \{4, 4\}$ then either $w(K, \frac {r}{s}) = \{ 4\}$ or $w(K, \frac {r}{s})$ contains a 0, 1, 2, or 3.  If $w(K, \frac {r}{s}) = \{ 4\}$ then $K$ has a twice punctured $\frac {r}{s}$-sloped Heegaard torus $S$.  By Theorem \ref{theorem:cappingoff}  $S$ is capped off in $K( \frac {r}{s})$ to become a Heegaard torus $\hat{S}$, a contradiction.


If $w(K, \frac {r}{s})$ contains a 2 then there is some sub-manifold $W_i=C_i^1 \cup_T C_i^2$ with a Heegaard torus $T$.  By Lemma \ref{lemma:boundary+lessboundary}, both $\partial C_i^1$ and $\partial C_i^2$ must be of lower complexity than $T$, and have no more boundary components than $T$.  Thus they are both essential spheres in $K$, a contradiction.

 $w(K, \frac {r}{s})$ containing a 1 is ruled out by Theorem \ref  {theorem:1s}.  If $w(K, \frac {r}{s})$ contains a 0, then by Lemma \ref{lemma:boundary+lessboundary} K has a genus 0 Heegaard surface, a contradiction.  $w(K, \frac {r}{s}) = \{ 3\}$ is ruled out by Theorem  \ref{theorem:3s}.


$M$ has a unique essential annulus up to isotopy, so in a thin decomposition of $M$, any essential annulus is parallel to the cabling annulus $A$.   $A$ is capped off to $\hat{A}$, a sphere which realizes $K(\frac {r}{s})$ as the connect sum of two lens spaces.  If $A$ appears as the $i$th thin surface then the surfaces $\{\hat{F_j},\hat{S_j}\} i+1\leq j \leq n$ are capped off to be a generalized Heegaard splitting for one of the lens spaces $K(\frac {r}{s})| \hat{A}$.  

If $\{ 4\}<  w(K, \frac {r}{s}) <  \{4,4\}$ it must contain exactly one 4 and at least one 3, and no other numbers.  Only one of the thick surfaces in this decomposition (the one with complexity 4) has non zero genus.    By  Lemma \ref{lemma:thinclosed}, since all the thick surfaces have boundary,  all of the thin surfaces in the thin decomposition must have boundary.  Because of the requirements on the complexity of thin surfaces in a thin decomposition, each thin surface is either a single annulus, or a pair of annuli.

If a thin decomposition contains thin surface which is a single annulus $A=F_m$, $A$ is parallel to the cabling annulus, and the collections of surfaces $\{F_i, S_i\}1\leq i \leq m$ and $\{F_j, S_i\}m+1\leq j \leq n$ are each capped off in $K(\frac {r}{s})$ to become generalized Heegaard splittings for lens spaces.  In one of these lens spaces, all of the thick surfaces are spheres, a contradiction. Thus no thin surfaces are single annuli.   

If there is a thin surface which is a pair of annuli $A \cup A'$, each of $A$ and $A'$ must be a copy of the cabling annulus.  Thus $A \cup A'$ is not separating in $K$, and so cannot be a surface in a generalized boundary Heegaard splitting.

 If $\{ 3\}<  w(K, \frac {r}{s}) <  \{ 4\}$ then $w(K, \frac {r}{s}) $ is an $n$-tuple of 3s.  By the argument above, this is impossible.  Thus $w(K, \frac {r}{s}) \geq \{4, 4\}$.

\item If $|pqr+s|\neq1$ or 0,  then $w(K, \frac {r}{s}) = \{7\}$.\\

Since torus knots are tunnel number one, they have genus 2 Heegaard splittings, and thus by Theorem \ref{theorem:Hgenus} $w(K, \frac {r}{s}) \leq \{7\}$ for any $\frac {r}{s}$.    Assume $w(K, \frac {r}{s}) < \{7\}$.  This means that $K$ contains a thick surface $S_j$ of complexity less than 7.   If $S_j$ is a boundary Heegaard surface, then it a $\frac {r}{s}$ sloped punctured surface of genus at most 1.  $S_j$ is capped off to become a Heegaard surface for $w(K, \frac {r}{s})$, of genus at most 1.   But $K( \frac {r}{s})$ is a Seifert fibered space over a sphere with 3 exceptional fibers, which has Heegaard genus two, and width $\{4\}$, a contradiction.  If the thin $\frac {r}{s}$ sloped decomposition has thin surfaces, then  in the induced closed splitting of $K( \frac {r}{s})$ there are more than one thick surface, each of which has complexity at most 4, contradicting the width of $K( \frac {r}{s})$.  Thus  $w(K, \frac {r}{s}) = \{7\}$.  

\item $w(K, \emptyset)=\{5\}$. \\

As above, torus knots are tunnel number one.  Hence they have  minimal genus  two Heegaard splittings, so $w(K, \frac {r}{s}) = \{5\}$.  

\end{enumerate}
\end{proof}

\begin{lemma} Let $K \subset S^3$ be a non-trivial $(p,q)$ torus knot and $\frac {r}{s} \epsilon \mathbb{Q}$.   If $\{F_i,S_i\}$ is a $\frac {r}{s}$-thin decomposition of $K$, then the induced decomposition, $\{\hat{F_i},\hat{S_i}\}$, is a thin decomposition of $K(\frac {r}{s})$.  
\begin{proof}
\begin{enumerate}
\item $|pqr+s|=1$. $K(\frac {r}{s})$ is the lens space $L_{|q|,ps^2}$.  The boundary-Heegaard surface $S$ is a twice punctured torus, which is capped off in $K(\frac {r}{s})$ to be a Heegaard torus.  Since $L_{|q|,ps^2}$ has Heegaard genus 1, the splitting is thin.  

\item  $|pqr+s|=0$.   $K(\frac {r}{s})$ is the connect sum of the lens spaces: $L_{r,s}\# L_{s,r}$.  The thin surface $F_1$ is the cabling annulus for $K$, and is capped off to a sphere which bounds a punctured lens space on each side.  The thick surfaces $S_1$ and $S_2$ are twice punctured tori, which are each capped off to be Heegaard tori for the respective lens space.  The width of this decomposition is $\{2,2\}$, which is $w(L_{r,s}\# L_{s,r}, \emptyset)$, and thus this is a thin decomposition.  

\item   $|pqr+s|\neq1$ or 0.  $K(\frac {r}{s})$ is a Seifert fibered space over a sphere with 3 exceptional fibers.  The boundary Heegaard surface $S$ described above is a twice punctured genus two surface, and is capped off to be a genus two Heegaard surface.   $w(L_{r,s}\# L_{s,r})$ has Heegaard genus two, and since weakly reducible Heegaard splittings of genus two are reducible, $w(L_{r,s}\# L_{s,r}, \emptyset)=\{4\}$, $\hat{S}$ is a thin decomposition.  
\end{enumerate}
\end{proof}
  
\end{lemma}

%
%
%
%
%
%
%

    \bibliographystyle{hplain}
    \bibliography{../References/MooreReferences}

\begin{thebibliography}{10}

\bibitem{BerSKWSYLS}
John Berge.
\newblock Some knots with surgeries yielding lens spaces.
\newblock Preprint.

\bibitem{BonOtaSdHdEL}
Francis Bonahon and Jean-Pierre Otal.
\newblock Scindements de {H}eegaard des espaces lenticulaires.
\newblock {\em Ann. Sci. \'Ecole Norm. Sup. (4)}, 16(3):451--466 (1984), 1983.

\bibitem{BoyDSoK}
Steven Boyer.
\newblock Dehn surgery on knots.
\newblock In {\em Handbook of geometric topology}, pages 165--218.
  North-Holland, Amsterdam, 2002.

\bibitem{CasGorRHS}
A.~J. Casson and C.~McA. Gordon.
\newblock Reducing {H}eegaard splittings.
\newblock {\em Topology Appl.}, 27(3):275--283, 1987.

\bibitem{CulGorLueShaDSoK}
Marc Culler, C.~McA. Gordon, J.~Luecke, and Peter~B. Shalen.
\newblock Dehn surgery on knots.
\newblock {\em Ann. of Math. (2)}, 125(2):237--300, 1987.

\bibitem{GorLueKADBTC}
C.~McA. Gordon and J.~Luecke.
\newblock Knots are determined by their complements.
\newblock {\em Bull. Amer. Math. Soc. (N.S.)}, 20(1):83--87, 1989.

\bibitem{GunKWDPPaPSR}
Brandy Guntel.
\newblock Knots with distinct primitive/primitive and primitive/seifert
  representatives.
\newblock arXiv:0909.0476.

\bibitem{HatOtBCoIS}
A.~E. Hatcher.
\newblock On the boundary curves of incompressible surfaces.
\newblock {\em Pacific J. Math.}, 99(2):373--377, 1982.

\bibitem{HayShiTPoP3M1S}
Chuichiro Hayashi and Koya Shimokawa.
\newblock Thin position of a pair (3-manifold, 1-submanifold).
\newblock {\em Pacific J. Math.}, 197(2):301--324, 2001.

\bibitem{JacLo3MT}
William Jaco.
\newblock {\em Lectures on 3-manifold topology}, volume~43 of {\em CBMS
  Regional Conference Series in Math}.
\newblock Amer. Math. Soc., Providence, RI, 1980.

\bibitem{JohThoTN1KWaN1n}
Jesse Johnson and Abigail Thompson.
\newblock Tunnel number one knots which are not (1,n).
\newblock 2006, math.GT/0606226.

\bibitem{MorSedTHSoDFM}
Yoav Moriah and Eric Sedgwick.
\newblock The {H}eegaard structure of {D}ehn filled manifolds.
\newblock In {\em Workshop on {H}eegaard {S}plittings}, volume~12 of {\em Geom.
  Topol. Monogr.}, pages 233--263. Geom. Topol. Publ., Coventry, 2007.

\bibitem{MosESAaTK}
Louise Moser.
\newblock Elementary surgery along a torus knot.
\newblock {\em Pacific J. Math.}, 38:737--745, 1971.

\bibitem{RieSedPoHSUDF}
Yo'av Rieck and Eric Sedgwick.
\newblock Persistence of {H}eegaard structures under {D}ehn filling.
\newblock {\em Topology Appl.}, 109(1):41--53, 2001.

\bibitem{SchaThoTPf3M}
Martin Scharlemann and Abigail Thompson.
\newblock Thin position for 3-manifolds.
\newblock {\em Contemporary Mathematics}, 164, 1994.

\bibitem{ThuTDGaT1}
William~P. Thurston.
\newblock {\em Three-dimensional geometry and topology. {V}ol. 1}, volume~35 of
  {\em Princeton Mathematical Series}.
\newblock Princeton University Press, Princeton, NJ, 1997.
\newblock Edited by Silvio Levy.

\bibitem{TomTPfKi3M}
Maggy Tomova.
\newblock Thin position for knots in a 3-manifold.
\newblock {\em J. Lond. Math. Soc. (2)}, 80(1):85--98, 2009.

\end{thebibliography}

\end{document}